\documentclass[11pt,a4paper,oneside]{article}
\newcommand{\ignore}[1]{}
\normalbaselines
\usepackage{amsmath}
\usepackage{amssymb}
\usepackage{appendix}
\usepackage{mathtools}
\usepackage{mathcomp}
\usepackage{textcomp}
\usepackage{amsthm}
\usepackage{enumerate}
\usepackage[auth-lg]{authblk}
\usepackage{cite}
\usepackage{tikz}
\usepackage{caption}
\usepackage{subcaption}
\usepackage{bm}
\usetikzlibrary{arrows}
\usepackage{multicol}
\usepackage{tikz-qtree}
\usepackage{rotating}
\usepackage{url}
\usepackage{array}

\newtheorem{thm}{Theorem}[section]\theoremstyle{plain}
\newtheorem{theorem}[thm]{Theorem}\theoremstyle{plain}
\theoremstyle{plain}
\newtheorem{lemma}[thm]{Lemma}\theoremstyle{plain}
\theoremstyle{plain}

\theoremstyle{plain}
\newtheorem{claim}[thm]{Claim}\theoremstyle{plain}
\newtheorem{corollary}[thm]{Corollary}\theoremstyle{plain}
\theoremstyle{plain}
\theoremstyle{plain}
\newtheorem{conjecture}[thm]{Conjecture}\theoremstyle{plain}

\theoremstyle{remark}
\newtheorem{example}{Example}

\usepackage{bm}
\DeclareMathOperator{\rank}{rank}

\DeclareMathOperator{\cl}{cl}

\DeclareMathOperator{\cyc}{cyc}

\newcommand{\be}{{\bm e}}

\newcommand{\val}{{\rm val}}

\newcommand{\X}{{\cal X}}

\newcommand{\MS}{{\mathcal S}}
\newcommand{\MI}{{\mathcal I}}

\newcommand{\MC}{{\mathcal C}}

\newcommand{\MW}{{\mathcal W}}
\newcommand{\MX}{{\mathcal X}}

\newcommand{\MZ}{{\mathcal Z}}
\newcommand{\ML}{{\mathcal L}}

\newcommand{\MR}{{\mathcal R}}

\newcommand{\ZZ}{{\mathbb Z}}
\newcommand{\R}{{\mathbb R}}

\newcommand{\sm}{{\setminus}}

\begin{document}

\title{Symmetric Tensor Matroids, Dual Rigidity Matroids, and the Maximality Conjecture}
\author{Bill Jackson and\\ Shin-ichi Tanigawa}
\author{ 
Bill Jackson\thanks{School
of Mathematical Sciences, Queen Mary University of London,
Mile End Road, London E1 4NS, England. email:
{\tt B.Jackson@qmul.ac.uk}}  
 and Shin-ichi Tanigawa\thanks{Department of Mathematical Informatics, Graduate School of Information Science and Technology, University of Tokyo, 7-3-1 Hongo, Bunkyo-ku, 113-8656,  Tokyo Japan. email: {\tt tanigawa@mist.i.u-tokyo.ac.jp}}}

\date{\today}

\maketitle

\abstract{
Inspired by a recent result of Brakensiek et al.~\cite{brakensiek} that  symmetric tensor matroids and rigidity matroids are linked by matroid duality, we define abstract symmetric tensor matroids as a dual concept to abstract rigidity matroids and establish their basic properties. We then exploit this duality to obtain an alternative characterisation of the generic $d$-dimensional  rigidity on $K_n$ for $n-d\leq 6$ to that given by Grasseger et al.~in \cite{GGJN}. Our results imply that Graver's maximality conjecture  holds for these matroids. We also consider the related family of $K_{1,t+1}$-matroids on $K_n$ and show that this family has a unique maximal element only when $t\leq 3$. 
This implies that the family of second quasi symmetric  powers of the uniform matroid $U_{t,n}$ does not have a unique maximal matroid if $t\geq 4$ and $n$ is sufficiently large.
}
\section{Introduction}
The following three geometric matroids on the edge set of a graph 
$G=(V,E)$ 
are special cases of constructions given by Lov\'asz \cite{L} in 1977.  

\begin{itemize}
\item The {\em skew-tensor matroid ${\cal W}_t(G)$}. This is the
linear matroid on $E$ obtained by choosing a generic realisation $p:V\to \R^t$ 
of the vertices of $V$ in $t$-dimensional Euclidean space 
and then  
 associating each edge $ij\in E$ with the vector $p_i\wedge p_j$.
\item The {\em symmetric-tensor matroid ${\cal S}_t(G)$}. This is the
linear matroid on $E$ obtained by choosing a generic realisation $p:V\to \R^t$  and then  
 associating each edge $ij\in E$
 with the vector $p_i\otimes p_j+p_j\otimes p_i$.

\item The {\em tensor matroid ${\cal T}_{t_1,t_2}(G)$}. For this matroid, we assume $G$ is bipartite with bipartition $(X,Y)$ and choose generic realisations $p:X\to \R^{t_1}$ and $q:Y\to \R^{t_2}$. We then  
 associate each edge $ij\in E$ with the vector
 $p_i\otimes q_j$.
\end{itemize}
Note that each of these matroids are uniquely defined by the graph $G$ because of the genericity of the realisations.   

These matroids are closely linked to matroids which have been extensively studied in combinatorial rigidity theory -- we refer the reader to  \cite{W} for a comprehensive survey of this theory. The first such connection was obtained by Izmestiev \cite{Iz} who showed that, if $F$ is an edge cut of a graph $G$ whose edge set is a circuit in the generic $d$-dimensional rigidity matroid, then $F$ is the disjoint union of circuits of $\MW_d(G)$. 
A recent paper by Brakensiek et al.~\cite{brakensiek}~gives a much closer connection.

\begin{theorem}[Brakensiek et al.~\cite{brakensiek}]\label{thm:dual}
Let $n, n_1, n_2, d, d_1, d_2$ be integers.
\begin{itemize}
\item For all $1\leq d\leq n$, the $d$-dimensional hyperconnectivitiy matroid ${\cal H}_{d}(K_n)$ 
is the dual of ${\cal W}_{n-d}(K_n)$.
\item For all  $1\leq d\leq n$, the $d$-dimensional symmetric-matrix completion matroid ${\cal I}_{d}(K_n^{\circ})$ 
is the dual of ${\cal S}_{n-d}(K_n^{\circ})$, where $K_n^{\circ}$ denotes the {looped complete graph} i.e. the graph obtained from $K_n$ by adding a loop at each vertex.
\item For all  $1\leq d\leq n+1$, the $d$-dimensional rigidity matroid ${\cal R}_{d}(K_n)$ 
is the dual of ${\cal S}_{n-d-1}(K_n)$.
(Note that the $d$-dimensional rigidity matroid ${\cal R}_{d}(K_n)$
can be obtained from ${\cal I}_{d+1}(K_n^{\circ})$ by contracting all its loops, see, e.g., \cite{JJT}.)
\item For all  $1\leq d_1\leq n_1$ and $1\leq d_2\leq n_2$, 
the $(d_1,d_2)$-dimensional birigidity matroid ${\cal B}_{d_1,d_2}(K_{n_1,n_2})$ 
is the dual of ${\cal T}_{n_1-d_1,n_2-d_2}(K_{n_1,n_2})$.
\end{itemize}
\end{theorem}
The hyperconnectivitiy matroid, symmetric-matrix completion matroid  and birigidity matroid were introduced by Kalai \cite{K}, Singer and Cucuringu \cite{SC} and  Kalai, Nevo and Novik \cite{KNN}, respectively. A brief overview of all three matroids as well as the rigidity matroid  can be found in \cite[Section 6]{JT}.

In this paper, we will explore links between  the matroids $\MS_t$ and $\MR_d$. (We will obtain analogous results for other matroids in a follow up paper.)
Inspired by Graver's concept of an abstract $d$-rigidity matroid \cite{G91}, we will use two fundamental properties of $\MS_t$ to define a family of `abstract symmetric $t$-tensor matroids' on $K_n$ and show that this family is indeed the dual of the family of abstract $d$-rigidity matroids on $K_n$ {\em when we take $d=n-t-1$}. We will use this duality to derive other properties shared by all abstract symmetric $t$-tensor matroids. We will also consider the dual version of Graver's maximality conjecture for abstract symmetric $t$-tensor  matroids and formulate a conjecture which would characterise the supposed unique maximal symmetric $t$-tensor matroid. We then turn our attention to  the particular abstract tensor matroid $\MS_t$. We characterise independence in this matroid for $1\leq t\leq 5$ and deduce that $\MS_t$ is the unique maximal abstract symmetric $t$-tensor  matroid for these values of $t$. Our condition for independence is a new kind of sparsity condition in which the the upper bound on the number of edges in a subgraph depends on the isomorphism class of the subgraph (not just its number of vertices). Finally we consider the larger family of $K_{1,t+1}$-matroids, i.e. matroids on the edge set of $K_n$ in which each copy of $K_{1,t+1}$ is a circuit, and show that this larger family has a unique maximal element only when $t\leq 3$.


\subsubsection*{Example} 
The family of abstract 1-rigidity matroids on $K_n$ is the family of all matroids defined on  $E(K_n)$ which have rank $n-1$ and in which the edge set of every copy of $K_3$ is a circuit. The 1-dimensional rigidity matroid ${\cal R}_1(K_n)$ is the unique maximal abstract 1-rigidity matroid on $K_n$. (Indeed, in this 1-dimensional example, it is the unique abstract 1-rigidity matroid on $K_n$.) 
It is equal to the {\em cycle matroid} of $K_n$ in which a set of edges is a circuit if and only if it induces a cycle in $K_n$. The dual matroid of the cycle matroid is the {\em cocycle matroid} on $K_n$ in which a set of edges is a circuit if and only if it induces a minimal edge cutset in $K_n$. It is the unique (maximal) matroid in the family of abstract symmetric $(n-2)$-tensor matroids, which we define to be the family of all matroids on  $E(K_n)$ which have rank ${{n-1}\choose{2}}$ and in which the edge set of every copy of $K_{1,n-1}$ is a circuit. Theorem \ref{thm:dual} tells us that this matroid is equal to the symmetric tensor matroid
${\cal S}_{n-2}(K_n)$.

The cocircuits in the $d$-dimensional rigidity matroid have a significantly more complicated structure when $d \ge 2$. Nevertheless, the $K_{1,n-1}$-cocircuits of ${\cal R}_1$ extend naturally to $K_{1,n-d}$-cocircuits of ${\cal R}_d$.  This observation of Graver et al. \cite{GSS93} is a key property  of abstract rigidity matroids, and our definition of abstract symmetric tensor  is built on their observation.

\subsubsection*{Notation and Terminology}

Let $G=(V,E)$ be a graph. 
For $X\subseteq V$, let $G[X]$ be the subgraph of $G$ induced by $X$.
For $F\subseteq E$, let $V(F)$ be the set of vertices incident to $F$, and let $G[F]=(V(F),F)$.
For $v\in V$, let $N_G(v)$ be the set of neighbors of $v$ in $G$, and let $d_G(v)=|N_G(v)|$.
For $X=\{v_1,\dots, v_k\}\subseteq V$, let $K(X)$ or $K(v_1,\dots, v_k)$ be the edge set of the complete graph on $X$.
The {\em complement} of $G$ is the graph $\bar G$ with vertex set $V$ and edge set $K(V)\sm E$. 


Given two graphs $G$ and $H$ with no common vertices, we will use $G\dot\cup H$ to denote the disjoint union of $G$ and $H$, and $G+ H$ to denote the {\em join} of $G$ and $H$, i.e. the graph obtained from $G\dot\cup H$ by adding all edges between $G$ and $H$.

Given a matroid defined on the edge set of a complete graph $K_n$ and a graph $G\subseteq K_n$, we will often simplify terminology and say that $G$ has a certain property in $M$ to mean that $E(G)$ has this property in $M$. For example we say that $G$ is a base of $M$ to mean $E(G)$ is a base of $M$. We refer the reader to \cite{O} for matroid concepts not explicitly defined in this paper.


\section{Abstract Rigidity Matroids and Duality}\label{sec:rig}
Graver~\cite{G91} defined the family of abstract $d$-rigidity matroids using two fundamental properties of the generic $d$-dimensional rigidity matroid $\MR_d(K_n)$ in order to get a better understanding of this matroid.  
Theorem~\ref{thm:dual}(b) above motivates us to adopt a similar approach for the symmetric tensor matroid $\MS_t(K_n)$.
We will first describe the axioms for abstract rigidity and then dualise them to obtain our family of `abstract symmetric tensor matroids'.

\subsection{Abstract rigidity matroids}

An {\em abstract $d$-rigidity matroid} is
a matroid $M$ on the edge set of the complete graph $K_n$ with $n\geq d+1$, whose closure operator ${\rm cl}_M$ satisfies the following two axioms.
\begin{description}
\item[(G1)] If $E_1, E_2\subseteq E(K_n)$ with $|V(E_1)\cap V(E_2)|\leq d-1$, then
${\rm cl}_M(E_1\cup E_2)\subseteq K(V(E_1))\cup K(V(E_2))$;
\item[(G2)] If $E_1, E_2\subseteq E(K_n)$ with ${\rm cl}_M(E_1)=K(V(E_1))$, ${\rm cl}_M(E_2)=K(V(E_2))$, and $|V(E_1)\cap V(E_2)|\geq d$, then
${\rm cl}_M(E_1\cup E_2)= K(V(E_1\cup E_2))$.
\end{description}
These conditions reflect two fundamental rigidity properties of generic $d$-dimensional bar-joint frameworks. 
Other examples of abstract $d$-rigidity matroids on $K_n$ can be obtained by taking the bar-joint rigidity matroid of any general position realisation of $K_n$ in $\R^d$ or by considering the generic $C^{d-2}_{d-1}$-cofactor matroid of $K_n$, see \cite{W}. 


Equivalent characterizations of abstract $d$-rigidity matroids were given by  Graver, Servatius and Servatius in \cite{GSS93} and Nguyen~\cite{N10}.
To describe these characterisations, we need to consider four more properties of a matroid $M$ on $E(K_n)$, with $n\geq d+1$.
\begin{description}
\item[(R1)] The rank of $M$ is $dn-{d+1\choose 2}$.
\item[(R2)] Every copy of $K_{d+2}$ is a circuit.
\item[(R3)] Every copy of $K_{1,n-d}$ is a cocircuit.
\item[(R4)] 
If $G\subseteq K_n$, $v$ is a vertex of $G$ of degree $d$ and $G-v$ is independent in $M$, then $G$ is independent in $M$.
\end{description}
\begin{theorem}[Graver-Servatius-Servatius\cite{GSS93} and Nguyen~\cite{N10}]\label{thm:hang}
Let $n,d$ be positive integers with $n\geq d+1$ and $M$ be a matroid on $E(K_n)$.
Then, the following statements are equivalent.
\begin{itemize}
    \item $M$ is an abstract $d$-rigidity matroid.
    \item $M$ satisfies (R1) and (R2).
    \item $M$ satisfies (R1) and (R3).
    \item $M$ satisfies (R2) and (R3).
    \item $M$ satisfies (R2) and (R4).
\end{itemize}
\end{theorem}

The fact that every abstract $d$-rigidity matroid satisfies (R3) 
can be extended as follows.
\begin{lemma}\label{lem:Kab_cocircuit}
Let $n,d$ be positive integers with $n\geq d+2$ and $M$ be 
an abstract $d$-rigidity matroid on $K_n$. Then
every copy of $K_{a,b}$ with $a,b\geq 1$ and $a+b=n-d+1$
is a cocircuit in $M$.
\end{lemma}
\begin{proof}
Consider two complete subgraphs $G_1$ and $G_2$ with $|V(G_1)|=a+d-1$, $|V(G_2)|=b+d-1$, and $|V(G_1)\cap V(G_2)|=d-1$.
Then, by the fact that the rank of $G_i$ is $d|V(G_i)|-{d+1\choose 2}$, one can check that the rank of $G_1\cup G_2$ is at least $dn-{d+1\choose 2}-1$. Property
(G1) further implies that $G_1\cup G_2$ is closed.
Since the rank of $M$ is $dn-{d+1\choose 2}$, $G_1\cup G_2$ must be a hyperplane.
Since the complement of $E(G_1\cup G_2)$ forms the edge set of $K_{a,b}$ with $a+b=n-d+1$, the statement follows.
\end{proof}
For $d=2$, Lov{\'a}s and Yemini~\cite{LY} gave a combinatorial characterization of the closure of each set in the generic $2$-dimensional rigidity matroid ${\cal R}_2(K_n)$.
Using this, Servatius~\cite{Servatius} gave a combinatorial characterization of cocircuits  of ${\cal R}_2(K_n)$.

\subsection{Abstract Symmetric Tensor Matroids}
We will define the family of abstract symmetric $t$-tensor matroids by
dualizing the properties in the last subsection and putting $t=n-d-1$.
We begin by choosing three properties (R1), (R2) and (R3) which are easy to dualise. Given integers $t, n$ with $n\geq t+1$, consider the following dual properties for a matroid $M$ on $E(K_n)$.
\begin{description}
\item[(R1*)] The rank of $M$ is ${t+1\choose 2}$.
\item[(R2*)] Every copy of $K_{n-t+1}$ is a cocircuit in $M$.
\item[(R3*)] Every copy of $K_{1,t+1}$ is a circuit in $M$.
\end{description}

We may apply matroid duality to Theorem~\ref{thm:hang} to deduce that any two of these properties are equivalent to the statement that the dual matroid $M^*$ is an abstract $d$-rigidity matroid for a suitable choice of $d$.  

\begin{lemma}\label{lem:dual}
Let $n,t$ be non-negative integers  with $n\geq t+1$
and $M$ be a matroid on $E(K_n)$. 
Then, the following statements are equivalent.
\begin{itemize}
    \item $M$ satisfies (R1*) and (R2*).
\item $M$ satisfies (R1*) and (R3*).
\item $M$ satisfies (R2*) and (R3*).
\item $M^*$ is an abstract $(n-t-1)$-rigidity matroid on $K_n$.
\end{itemize}
\end{lemma}
\begin{proof}
This is a dual interpretation of Theorem~\ref{thm:hang} for $d=n-t-1\geq 0$.
Note that ${n\choose 2}-\left(dn+{d+1\choose 2} \right)={n-d\choose 2}={t+1\choose 2}$.
\end{proof}
In view of Theorem~\ref{thm:dual} and Lemma~\ref{lem:dual},
we say that a matroid $M$ on $E(K_n)$ is an {\em abstract symmetric $t$-tensor matroid}
if it satisfies properties (R1*) and (R3*).
Thus $M$ is an abstract symmetric $t$-tensor matroid on $E(K_n)$  if and only if its dual matroid is an abstract symmetric $(n-t-1)$-tensor matroid on $E(K_n)$, and the symmetric $t$-tensor matroid on $E(K_n)$ belongs to the family of  abstract symmetric $t$-tensor matroids on $E(K_n)$.

Since symmetric $t$-tensor matroids are defined on the looped complete graph $K_n^{\circ}$, we can define 
abstract symmetric $t$-tensor matroids on $E(K_n^{\circ})$ using the same axioms.
Then, the dual provides an abstraction of symmetric-matrix completion matroids. We will give more details on this family in a follow up paper.

By definition, the family of abstract symmetric $t$-tensor matroids on $E(K_n^{\circ})$ coincides with the  family of second symmetric powers of the uniform matroid $U_{t,n}$ in the sense of Lov{\'a}sz~\cite{L} and Mason~\cite{M}. In addition, they refer to a matroid satisfying (R3*), but not necessarily (R1*) or (R2*), as a second quasi symmetric power of $U_{t,n}$. 

\medskip

Dualising Lemma~\ref{lem:Kab_cocircuit} immediately gives us the  following result.
\begin{lemma}\label{lem:dual_circuits}
Let $M$ be  an abstract symmetric $t$-tensor matroid on $E(K_n)$ with $n\geq t+1$.
Then every
copy of $K_{a,b}$ with $a,b\geq 1$ and $a+b=t+2$
is a circuit of $M$.
\end{lemma}

Note that a  matroid  on $K_n$ in which every copy of $K_{a,b}$ with $a,b\geq 1$ and $a+b=t+2$
is a circuit may not be an abstract symmetric $t$-tensor matroid. To see this, we can just take the truncation of any abstract symmetric $t$-tensor matroid on $K_n$. This operation will not change any of the circuits $K_{a,b}$ but will reduce the rank to $t(t+1)-1$.

Basic properties of abstract rigidity matroids can give rise to  non-trivial properties of  abstract symmetric tensor matroids.
Our next two lemmas, which determine families of sets which are cyclic flats in every abstract symmetric tensor  matroid, are  examples of this. (Recall that a set is {\em closed} in a matroid  if it is equal to its own closure and is {\em cyclic} if it is the union of cicuits.) 

\begin{lemma}\label{lem:cyclic_flats}
Let $n,t$ be positive integers with $n \geq t+1$ and  $M$ be an abstract symmetric $t$-tensor matroid on $K_n$. 
\begin{description}
\item[(a)] $K_{a,b}$ is a cyclic flat of $M$ with rank $ab-{a+b-t\choose 2}$ whenever $3\leq a\leq t-1$, $3\leq b\leq  t-1$ and $t+2\leq a+b\leq n$.
In particular, if $a+b=t+2$, then $K_{a,b}$ is a closed circuit in $M$.
\item[(b)] $K_a+\bar{K}_{n-a}$ is a cyclic flat in $M$
with rank $at-{a\choose 2}$ in $M$ whenever $0\leq a\leq t-1$.
\end{description}
\end{lemma}
\begin{proof}
Let $M^*$ be the dual matroid of $M$ and put $d=n-t-1$. Then $M^*$ is an abstract $d$-rigidity matroid on $K_n$ by Lemma \ref{lem:dual}.\\
(a) 
Let $h=n-(a+b)$,
and let $H$ be the union of $K_{a+h}$ and $K_{b+h}$ with $h$ common vertices. Then the number of vertices of $H$ is $n$.
Since $a, b\leq t-1$ and $a+b\geq t+2$, we have 
$a+h\geq d+2$, $b+h\geq d+2$ and  $h\leq d-1$.
This implies that $K_{a+h}$ and $K_{b+h}$ are cyclic flats in $M^*$ by (R2) and hence $H$ is a cyclic flat in $M^*$ by (G1).
Thus    $\bar{H}=K_{a,b}$ is a cyclic flat in $M$.
To determine the rank of $K_{a,b}$ in $M$, observe that the rank deficiency of $H$ in $M^*$ is ${d+1-h\choose 2}$, which is equal to ${d+b-t\choose 2}$.
Hence, by the dual rank formula, the rank of $\bar{H}=K_{a,b}$ in $M$
is $ab-{d+b-t\choose 2}$.
\\[1mm]
(b) The complement of the edge set of  $K_a+\bar{K}_{n-a}$ is the edge set of $K_{n-a}$, which is a cyclic flat in $M^*$ since $n-a\geq n-t+1\geq d+2$. Hence its complement is a cyclic flat in $M$.
\end{proof}

\begin{lemma}\label{lem:abconing}
Let $n,t$ be positive integers with $2\leq t\leq n-2$, $M$ be an abstract symmetric $t$-tensor matroid on $K_n$ and $E_v$ be the set of edges of $K_n$ incident with a given vertex $v$. Then $N=M/E_v$ is an abstract symmetric $(t-1)$-tensor matroid on $K_{n-1}$. In addition, if $F$ is a cyclic flat of $N$ of rank $k$, then $F\cup E_v$ is a cyclic flat of $M$ of rank $k+t$.
\end{lemma}
\begin{proof}
 Since $M$ is an abstract symmetric $t$-tensor matroid on $K_n$, $M^*$ is an abstract $(n-t-1)$-rigidity matroid on $K_n$. This implies that $M^*-E_v$ is an abstract $(n-t-1)$-rigidity matroid on $K_{n-1}$ and hence 
 $(M^*-E_v)^*=M/E_v=N$ is an  abstract symmetric $(t-1)$-tensor matroid on $K_{n-1}$. 

Since $F$ is a cyclic flat in $N$, $\overline F$ is a cyclic flat in $N^*$.
Observe that  $\overline{F\cup E_v}$ is obtained from $\overline F$ by adding $v$ as an isolated vertex.
Hence $\overline{F\cup E_v}$ is a cyclic flat in $M^*$.
By duality this implies that $F\cup E_v$ is a cyclic flat in $M$.

The same trick combined with the rank formula for the  dual of a matroid
tells us that the rank of $F\cup E_v$ in $M$ is $k+t$.
\end{proof}

{For all positive integers $n,t$ 
we  recursively define a family ${\cal C}_{n,t}$ of graphs with $n$ vertices. 
We first put ${\cal C}_{n,t}=\{\bar K_n\}$ for all $t\geq 1$ and $n\leq t+1$ and 
${\cal C}_{n,1}=\{K_n,\bar{K}_n\}$ for all $n\geq 3$.} Then, for $t\geq 2$ and $n\geq t+2$, we put 
\[
\begin{split}
{\cal C}_{n,t}&=\{\bar K_n\}
\cup \{G+K_1:G\in {\cal C}_{n-1,t-1}\} \\
&\hspace{2em}\cup \{K_{a,b}\dot\cup \bar K_{n-a-b}: 3\leq a, b\leq t-1, t+2\leq a+b\leq n\}.
\end{split}
\]
We next recursively define a weight function $c_{t}:\bigcup_{n\geq 1}{\cal C}_{n,t}\to {\mathbb Z}$. We first  
put $c_t(\bar K_n)=0$ for $n,t\geq 1$ and $c_{1}(K_n)=1$ for all $n\geq 3$. We then let
$c_t(K_{a,b}\dot\cup \bar{K}_{n-a-b})=ab-{a+b-t\choose 2}$ for $3\leq a,b\leq t-1$ and $t+2\leq a+b\leq n$, and  
$c_t(G+K_1)=c_{t-1}(G)+t$ for all $G\in {\cal C}_{n-1,t-1}$ with $t\geq 2$ and $n\geq t+2$.

Observe that ${\cal C}_{n,t}$ can be explicitly written as 
\[
\begin{split}
{\cal C}_{n,t}=\left\{K_s+\bar{K}_{n-s}: 0\leq s\leq t-1\right\}\cup 
\left\{ K_s+(K_{a,b}\dot\cup \bar{K}_{n-(a+b+s)}): 
\begin{matrix} s\geq 0, a,b\geq 3, \\
a+b+s\geq t+2, \\ 
a+s\leq t-1, \\
b+s\leq t-1. 
\end{matrix}
\right\}
\end{split}
\]
where $K_0$  is taken to be the empty graph. 
Lemmas~\ref{lem:cyclic_flats} and \ref{lem:abconing} imply that, for each abstract symmetric $t$-tensor matroid $M$ on $K_n$, every $F\in {\cal C}_{n,t}$ is a cyclic flat in $M$ and $r_M(F)=c_t(F)$.
We can use this observation  to obtain a necessary condition for independence in any abstract symmetric $t$-tensor matroid. 

We
say that a graph $G$ on $n$ vertices is {\em ${\cal C}_{t}$-independent}
if, for 
all subgraphs $H$ of $G$ 
and all $F\in {\cal C}_{n,t}$ with $H\subseteq F$, we have
\[
|E(H)|\leq c_t(F), 
\]
where $H\subseteq F$ is taken to mean that $F$ has a subgraph isomorphic to $H$. 

\begin{lemma}\label{lem:necind}
Let $n,t$ be positive integers
with $1\leq t\leq n-1$, $M$ be an abstract symmetric $t$-tensor matroid on $K_n$ and $G$ be a graph on $n$ vertices which is independent in $M$. Then  $G$ is ${\cal C}_{t}$-independent.
\end{lemma}
\begin{proof} 
   Suppose, for a contradiction, that $G$ is not  ${\cal C}_{t}$-independent.   Then there exists $H\subseteq G$ 
   and $F\in {\cal C}_{n,t}$ with $H\subseteq F$ and  
$|E(H)|> c_t(F)=r_{M}(F)\geq r_{M}(H)$. This implies that  $H$ is dependent in $M$ and hence 
$G$ is also dependent in $M$. 
   \end{proof}

We will show in Theorem \ref{thm:characterization} below, that ${\cal C}_{t}$-independence is equivalent to ${\cal S}_t$-independence  when $1\leq t\leq 5$. Our proof uses the following two results. {The first 
characterises when the
converse of Lemma \ref{lem:necind} holds.

\begin{lemma}\label{lem:suffind}
Let $n,t$ be positive integers with $2\leq t\leq n-1$ and $M$ be an abstract symmetric $t$-tensor matroid on $K_n$. Then the following two statements are equivalent.\\
(a)  Every ${\cal C}_{t}$-independent graph $G$ on $n$ vertices is independent in $M$.\\
(b) $\cl_M(C)\in \MC_{n,t}$ for all circuits $C$ of $M$.
\end{lemma}
\begin{proof} 
   Suppose, (a) is true. Let $C$ be a circuit of $M$. Then $C$ is not $\MC_t$-independent so there exists $H\subseteq C$ and an $F\in \MC_{n,t}$ such that $H\subseteq F$  and $|E(H)|>c_t(F)=r_M(F)$. Lemma \ref{lem:necind} and the fact that every proper subgraph of $C$ is $M$-independent now tells us that $C=H$ and $r_M(C)=|E(C)|-1=r_M(F)$. Since $F$ is a flat of $M$, we have $\cl_M(C)=F\in \MC_{n,t}$. Hence (b) is true. 
   
   Suppose, (a) is false. Then there exists a $\MC_t$-independent graph $G$ on $n$ vertices which is not  independent in $M$. Let $C$ be a circuit of $M$ such that $C\subseteq G$ and put $F=\cl_M(C)$. Then $|E(C)|>r_M(C)=r_M(F)$. Since $G$ is   $\MC_t$-independent, $F\notin  \MC_{n,t}$.  Hence (b) is false.
   \end{proof}
}


Note that the idea of using the closures of all circuits to characterize independent sets generalizes Laman's count condition for independence in the generic two-dimensional rigidity matroid $\MR_2(K_n)$. This follows since ${\cal C}_n=\{K_a: 4\leq a\leq n\}$ is the list of  closures of all circuits in $\MR_2(K_n)$, their rank is given by $c_n(K_a)=2a-3$, and the corresponding notion of ${\cal C}$-independence is exactly Laman's count condition.

Our next result gives a lower bound on the degrees of minimally {\em $\MC_t$-dependent graphs}, i.e. graphs which are not $\MC_t$-independent  but are such that all their proper subgraphs 
are  $C_t$-independent.

\begin{lemma}\label{lem:mindep}
Suppose $G$ is a minimally $\MC_t$-dependent graph on $n$ vertices with at most ${t+1}\choose2$ edges for some $1\leq t\leq n-1$. Then\\  
(a) every edge of $G$ is incident with a vertex of degree at least two and\\
(b) $G$ has no isolated vertices.
\end{lemma}
\begin{proof}
  Since $G$ is not $\MC_t$-independent, there exists a subgraph $H$ of  $G$ and an $F\in \MC_{n,t}$ such that $H\subseteq F$ and $|E(H)|>c_t(F)$. The minimal dependency of $G$ now tells us that $G=H$ and $|E(G)|=c_t(F)+1$. This in turn tells us that $F\not\in \{K_n,\bar K_n\}$. 
  
The definition of ${\cal C}_{n,t}$ implies that either 
\begin{description}
    \item[(i)] $F=K_s+\bar{K}_{n-s}$ for some integer $s$ with $1\leq s\leq t-1$, or
    \item[(ii)] $F=K_s+(K_{a,b}\dot\cup \bar{K}_{n-(a+b+s)})$ for some integers $s, a, b$ with $s\geq 0$, $a,b\geq 3$, $a+b+s\geq t+2$, $a+s\leq t-1$, and $b+s\leq t-1$.
    \end{description}

Suppose (i) holds. Then $F=K_s+\bar{K}_{n-s}$ and  $|E(G)|=c_t(F)+1=st-{s\choose 2}+1$.
If $s=1$, then $G$ is isomorphic to $K_{1,t+1}$ and the lemma holds.
Hence we may assume $s\geq 2$.
Pick any vertex $v$ of $G$ in the $K_{s}$ subgraph of $F$.
Then $G-v$ is contained in a copy of $K_{s-1}+\bar{K}_{n-(s-1)}$. 
By the ${\cal C}_t$-independence of $G-v$ and $|E(G)|=c_t(F)+1$, we obtain $d_G(v)= |E(G)|-|E(G-v)|\geq c_t(K_{s}+\bar{K}_{n-s})+1-c_t(K_{s-1}+\bar{K}_{n-(s-1)})=st-{s\choose 2}+1 -(s-1)t-{s-1\choose 2}=t-s+2\geq 3$. Since $G\subseteq F$ and every edge of $F$ is incident to a vertex $v$ in the $K_s$-subgraph of $F$, this tells us that (a) holds for $G$. Part (b) follows similarly since, if $w$ is a vertex of $G$ in the $\bar K_{n-s}$ subgraph of $F$,
then $G-w$ is contained in a copy of $K_{s}+\bar{K}_{n-1-s}$ and 
the ${\cal C}_t$-independence of $G-w$ gives $d_G(w)= |E(G)|-|E(G-w)|\geq c_t(K_{s}+\bar{K}_{n-s})+1-c_{t}(K_{s}+\bar{K}_{n-s-1)})=1$.

Suppose (ii) holds. 
Then $F=K_s+(K_{a,b}\dot\cup \bar{K}_{n-(a+b+s)})$
and $|E(G)|=c_t(F)+1=ab-{a+b-(t-s)\choose 2}+st-{s\choose 2}+1$.
Observe that $F$ is an edge-disjoint union of $K_{a,b}$ and  $K_s+\bar{K}_{n-s}$.
Since $G\subseteq F$ and every edge of $F$ is incident to its $K_{a,b}$ or $K_s$ subgraph,
it suffices to show that each vertex of $G$ which belongs to the $K_{a,b}\dot \cup K_s$ subgraph of $F$ has degree at least two in $G$ and all other vertices have degree at least one. The latter assertion follows easily since, if $w$ is a vertex of $G$ which belongs to neither the $K_{a,b}$ nor the $K_s$ subgraph of $F$, then $G-w\subseteq F-w=K_s+(K_{a,b}\dot\cup \bar{K}_{n-1-(a+b+s)})$ and 
the ${\cal C}_t$-independence of $G-w$ gives $d_G(w)= |E(G)|-|E(G-w)|\geq 1$.
Consider the following two remaining cases.
\\[1mm]
Case 1: $a+b+s=t+2$.
Since $G$ is minimally $\MC_t$-dependent and $K_s+\bar{K}_{n-s}\in {\cal C}_{n,t}$,
the $K_s+\bar{K}_{n-s}$ subgraph of $F$ contains at most $st-{s\choose 2}+1$ edges of $G$. 
Since $a+b+s=t+2$, this implies that the $K_{a,b}$ subgraph contains at least $ab$ edges of $G$, and hence each vertex in this $K_{a,b}$ is incident to at least $\min\{a,b\}\geq 3$ edges of $G$.

Since the $K_{a,b}$ subgraph of $F$ can contain at most $ab$ edges of $G$, the
$K_s+\bar{K}_{n-s}$ subgraph contains at least $st-{s\choose 2}$ edges of $G$.
On other other hand, for any vertex $v$ of $G$ in the $K_{s}$ subgraph,
$G-v$ is contained in a copy of $K_{s-1}+\bar{K}_{n-(s-1)}$ and the ${\cal C}_t$-independence of $G-v$ implies that this 
$K_{s-1}+\bar{K}_{n-(s-1)}$ subgraph can contain 
at most $(s-1)t-{s-1\choose 2}$ edges of $G$.
Thus, $v$
is incident to at least $st-{s\choose 2}-\left((s-1)t-{s-1\choose 2}\right)=t-s+2\geq 3$ edges of $G$. 
This finishes the proof of Case 1.
\\[1mm]
Case 2: $a+b+s> t+2$.
Since $a+s\leq t-1$, $b+s\leq t-1$, and $a+b+s>t+2$, we have either $a\geq 4$ or  $b\geq 4$.
Without loss of generality, assume $a\geq 4$.
Then $(a-1)+b+s=a+b+(s-1)\geq t+2$ and
hence both $K_s+(K_{a-1,b}\dot\cup \bar{K}_{n-(a-1+b+s)})$
and $K_{s-1}+(K_{a,b}\dot\cup \bar{K}_{n-(a+b+(s-1))})$ belong to ${\cal C}_{n,t}$.
Let $v$ be a vertex of $G$ in the $K_{a,b}$ subgraph of $F$.
Then $G-v$ is contained in a $K_s+(K_{a-1,b}\dot\cup \bar{K}_{n-(a-1+b+s)})$,
and hence 
$d_G(v)\geq c_t(F)+1-c_t(K_s+(K_{a-1,b}\dot\cup \bar{K}_{n-(a-1+b+s)}))=t-s-a+2\geq 3$ by $a+s\leq t-1$.
On the other hand, if $v$ is a vertex of $G$ in the $K_s$-subgraph of $F$,
then $G-v$ is contained in a $K_{s-1}+(K_{a,b}\dot\cup \bar{K}_{n-(a+b+(s-1))})$,
and hence 
$d_G(v)\geq c_t(F)+1-c_t(K_{s-1}+(K_{a,b}\dot\cup \bar{K}_{n-(a+b+(s-1))}))=a+b+1\geq 7$.
This completes the proof of both Case 2 and the lemma.
\end{proof}

Lemma \ref{lem:mindep}(b) tells us that the $\MC_t$-independence of a graph $G$ is determined by its set of edges. In addition, if $\MC_t$-independence defines a matroid on $E(G)$, then the (edge-sets) of the minimally dependent subgraphs of $G$ will be the circuits in this matroid. 

\subsection{Maximality conjectures}
Graver's approach to abstract rigidity was to consider the family of abstract $d$-rigidity matroids on $K_n$ as a poset under the weak order of matroids. He conjectured in \cite{G91} that this poset has a unique maximal element, for any fixed $n$ and $d$, and that this maximal element is the generic $d$-dimension rigidity matroid $\MR_d(K_n)$. Graver verified his conjecture for $d=1,2$, but the second part of his conjecture was disproved for $d=4$ by N.~I.~Thurston, see \cite[Page 150]{GSS93}.  Subsequently  Whiteley \cite{W} disproved the conjecture for all $d\geq 4$  by showing that the generic cofactor matroid $\MC^{d-2}_{d-1}(K_n)\not \preceq \MR_d(K_n)$ when $d\geq 4$ and $n\geq 2d+4$. Whiteley offered the modified conjecture that $\MC^{d-2}_{d-1}(K_n)$ is the unique maximal abstract $d$-rigidity matroid on $K_n$ for all $n,d$. (It is easy to see that $\MC^{d-2}_{d-1}(K_n)= \MR_d(K_n)$ when $d=1,2$ and Whitely also conjectures that this identity extends to $d=3$.) Clinch et al.~\cite{CJT} recently verified Whiteley's maximality conjecture for $d=3$ by showing that $\MC^{1}_{2}(K_n)$ is the unique maximal abstract $3$-rigidity matroid on $K_n$ for all $n$.

Matroid duality immediately tells us that Graver's maximality conjecture is equivalent to the same conjecture for abstract symmetric $t$-tensor matroids. 

\begin{lemma}
Let $n,d$ be non-negative integers with $n\geq d+1$.
Then a matroid $M$ is the unique maximal abstract $d$-rigidity matroid on $E(K_n)$
if and only if its dual $M^*$ is the unique maximal abstract symmetric $t$-tensor matroid on $E(K_n)$, for $t=n-d-1$.
\end{lemma}
We will show in Theorem \ref{thm:characterization} below that $\MS_t$-independence  is equivalent to  $\MC_{t}$-independence  when $1\leq t\leq 5$. Together with Lemma \ref{lem:necind}, this implies that  the dual form of Graver's conjecture holds for all $1\leq t\leq 5$. 

The approach in \cite{CJT} is to consider an extension of the family of abstract $d$-rigidity matroids to the family of {\em $K_{d+2}$-matroids} i.e.~the family of all matroids on $K_n$ which satisfy (R2). They conjecture that $\MC^{d-2}_{d-1}(K_n)$ is the unique maximal $K_{d+2}$-matroid on $K_n$ for all $n,d$ and prove their conjecture for $d= 3$.

We can try a similar approach to the dual form of Graver's conjecture.
Since $K_{1,t+1}$ is a circuit in every abstract symmetric $t$-tensor matroid by (R3$^*$), it is tempting to ask whether the family of all $K_{1,t+1}$-matroids on $K_n$ has a unique maximal  element for all $n,t$. Theorem \ref{thm:K1t} below shows that this is false when $t\geq 4$. 
We believe however that the unique maximality property does hold if we use the circuits given by Lemma \ref{lem:dual_circuits} to define our family. 

\begin{conjecture}
Let $n,t$ be non-negative integers with $n\geq t+1$ and
let ${\cal X}=\{K_{a,b}: a,b\geq 1, a+b=t+2\}$.
Then there is a unique maximal ${\cal X}$-matroid on $K_n$.
\end{conjecture}

The remainder of this section is concerned with stating and proving Theorem \ref{thm:K1t}. 

\begin{theorem}\label{thm:K1t}
There exist two distinct maximal $K_{1,t+1}$-matroids on $K_{n}$ for all $t\geq 4$ and $n\geq t+2$.
\end{theorem}
\begin{proof} We refer the reader to the appendix for terminology not explicitly defined in this proof.
Let $M_0$ be the rank $t+1$ matroid on $K_{n}$ in which the set of non-spanning circuits is the set of all copies of $K_{1,t+1}$ in $K_{n}$ and let $M_1$ be the {\em free elevation} of $M_0$. It follows from \cite[Lemmas 2.7, 3.6]{JT} that $M_0$ is indeed a matroid and $\rank M_1\leq {{t+1}\choose2}$.  In addition, we may use  \cite[Lemmas 3.1, 3.8]{CJT} to deduce that  $M_1$ is a maximal  $K_{1,t+1}$-matroid on $K_{n}$ and every cyclic flat of $M_1$ is the union of copies of $K_{1,t+1}$.  Lemma \ref{lem:dual_circuits} and the hypothesis that $t\geq 4$ imply that $K_{3,t-1}$ is a closed circuit in every abstract symmetric $t$-tensor matroid. Since  $K_{3,t-1}$ does not contain a  copy of $K_{1,t+1}$,  this tells us that $M_1$ is not an abstract symmetric $t$-tensor matroid. Together with the fact that $M_1$ is a $K_{1,t+1}$-matroid and hence satisfies (R3$^*$), this implies that $M_1$ does not satisify (R1$^*$). Hence $\rank M_1<{{t+1}\choose2}$. This in turn implies that $M_1$ cannot be the unique maximal $K_{1,t+1}$-matroid on $K_{n}$
since $\MS_{t}(K_{n})$ is a $K_{1,t+1}$-matroid on $K_{n}$ of rank ${t+1}\choose2$.
\end{proof}

Further results on the matroid $M_1$ defined in the proof of Theorem \ref{thm:K1t} are given in the Appendix. 
We show that $M_1$ is the unique maximal $K_{1,t+1}$-matroid on $K_{n}$ when $t= 2,3$ and determine the lattice of cyclic flats in $M_1$ for all $t$ and all $n\geq 2\lfloor (2t+1)/3\rfloor$.

In response to a question of Mason~\cite{M}, Las Vergnas~\cite{Las} showed that the family of quasi tensor powers of $U_{3,4}$ has more than one maximal matroid.
Our result gives the symmetric analogue:
the family of second quasi symmetric  powers of $U_{t,n}$ has a unique maximal matroid if $t\leq 3$ and does not if $t\geq 4$ and $n\geq 2\lfloor (2t+1)/3\rfloor$.

\section{Exploring Duality between ${\cal R}_d$ and  ${\cal S}_t$}\label{sec:S}
In this section we introduce a graph operation for preserving ${\cal S}_t$-independence. This operation generalises several well known graph operations 
 for preserving ${\cal R}_d$-independence, such as 0-extension, 1-extension, and vertex-splitting, by the duality between ${\cal R}_d(K_n)$ and ${\cal S}_{n-d-1}(K_n)$ given in  Theorem~\ref{thm:dual}. 
We will then use this operation to characterise independence in ${\cal S}_t$ for $t\leq 5$.

\subsection{An operation for preserving ${\cal S}_t$-independence}

Let $G=(V,E)$ be a graph, $e_1, e_2,\dots, e_t\in E$.
We say that a vertex labelling $\ell:V\to \{1,2,\dots, t-1\}\cup\{\infty\}$ is {\em valid} with respect to $e_1, e_2,\dots, e_t$ if 
\begin{description}
\item[(i)] for each $e_i=u_iv_i$ with $i=1,\dots, t-1$, $i=\ell(u_i)<\ell(v_i)=\infty$, 
\item[(ii)]  for $e_t=u_tv_t$, $\ell(u_t)=\infty=\ell(v_t)$,
\item[(iii)] for other vertices $w$ (i.e., those not incident to any $e_i$), $\ell(w)=t-1$, and
\item[(iv)] for each $i=1,\dots, t$ and each endvertex $v_i$ of $e_i$ with $\ell(v_i)=\infty$,
the label of each neighbour of $v_i$ in $G-\{e_i, e_{i+1},\dots, e_t\}$ is less than  $i$.
\end{description}

The construction of $G=(V,E)$ from $G-\{e_1,\dots, e_t\}$ with a valid vertex labelling captures several well known graph operations from rigidity theory on the complementary graph $\bar G=(V,K(V)\sm E)$. 
We give some examples below.
\begin{example}
Let $v$ be a vertex of degree $t$ and $u_1,\dots, u_t$ be the neighbours of $v$.
Denote $e_i=vu_i$.
Then the labelling $\ell$ with $\ell(u_i)=i$ for $i=1,\dots, t-1$,
$\ell(u_t)=\ell(v)=\infty$, and $\ell(w)=t-1$ for other vertices $w$ is a valid labelling.
See Figure~\ref{fig:labelling}(a).
The complement of the construction of $G$ from $G-\{e_1,\dots, e_t\}$
corresponds to the ``$d$-dimensional 0-extension" with $d=n-t-1$, which constructs $\overline G$ from $\overline{G-v}$ by adding the vertex $v$ and $d$ edges from $v$ to $V-v$.
\end{example}
\begin{example}
 Let $v$ be a vertex of degree $t-1$, $u_1,\dots, u_t$ be the neighbours of $v$, $e_i=vu_i$ for $i=1,\dots, t-1$, and $e_t$ be an edge disjoint from $e_1,\dots, e_{t-1}$.
Denote the endvertices of $e_t$ by $u_t^1, u_t^2$.
Then the labelling $\ell$ with $\ell(u_i)=i$ for $i=1,\dots, t-1$,
$\ell(v)=\ell(u_t^1)=\ell(u_t^2)=\infty$ and $\ell(w)=t-1$ for all other vertices $w$ is a valid labelling.
See Figure~\ref{fig:labelling}(b).
The complement of the construction of $G$ from $G-\{e_1,\dots, e_t\}$
corresponds to the ``$d$-dimensional 1-extension operation" with $d=n-t-1$,  which constructs $\overline G$ from $\overline{G-v}$ by deleting $e_t$, adding the vertex $v$ and $d+1$ edges from $v$ to $u_t,v_t$ and $d-1$ other vertices of $V-v$.
\end{example}
\begin{example}
Suppose there are two non-adjacent vertices $v_1, v_2$ with $|N_G(v_1)\cup N_G(v_2)|=t$.
Denote $N_G(v_1)\cup N_G(v_2)=\{u_1, u_2,\dots, u_t\}$ such that 
$N_G(v_1)=\{u_1,\dots, u_{k_1}\}$
and $N_G(v_2)=\{u_{k_2} ,\dots u_t\}$,
and let $e_1=v_1u_1,\dots, e_{k_1}=v_1u_{k_1}, e_{k_1+1}=v_2u_{k_1+1}, \dots, e_t=v_2u_t$.
Then the labelling $\ell$ with $\ell(u_i)=i$ for $i=1,\dots, t-1$,
$\ell(u_t)=\ell(v_1)=\ell(v_2)=\infty$,
and $\ell(w)=t-1$ for all other vertices $w$ is a valid labelling.
See Figure~\ref{fig:labelling}(c,d).
The complement of the construction of $G$ from $G-\{e_1,\dots, e_t\}$
corresponds to the version of the  ``$d$-dimensional vertex splitting  operation" with $d=n-t-1$, which constructs  $\overline G$ from $\overline{G-v_1}$ by `splitting' the vertex $v_1$ away from $v_2$ and adding the edge $v_1v_2$. The valid vertex labelling which corresponds to the version of vertex splitting which does not add the edge $v_1v_2$ is illustrated in Figure~\ref{fig:labelling}(e).
\end{example}

Further examples for $t=5$ are illustrated in Figure~\ref{fig:labelling}(f,g,h).

\begin{figure}[t]
\centering
\begin{minipage}{0.32\textwidth}
\centering
\includegraphics[scale=0.55]{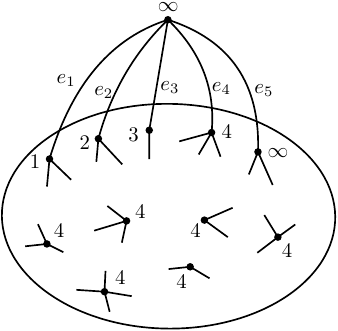}
\par
(a)
\end{minipage}
\begin{minipage}{0.32\textwidth}
\centering
\includegraphics[scale=0.55]{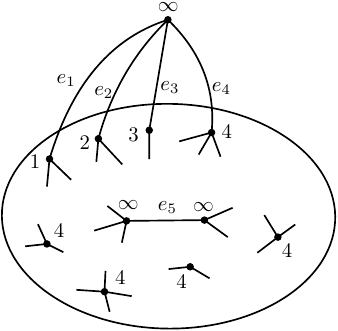}
\par
(b)
\end{minipage}
\begin{minipage}{0.32\textwidth}
\centering
\includegraphics[scale=0.55]{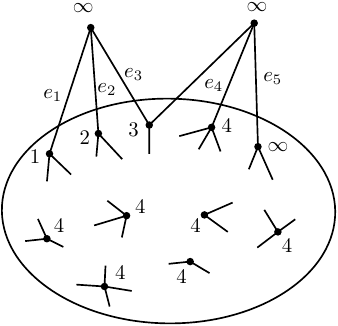}
\par
(c)
\end{minipage}
\begin{minipage}{0.32\textwidth}
\centering
\includegraphics[scale=0.55]{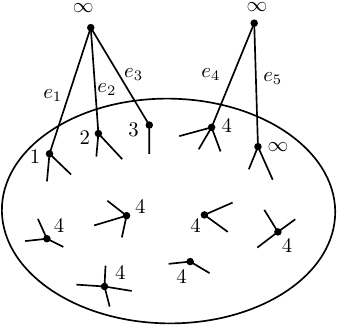}
\par
(d)
\end{minipage}
\begin{minipage}{0.32\textwidth}
\centering
\includegraphics[scale=0.55]{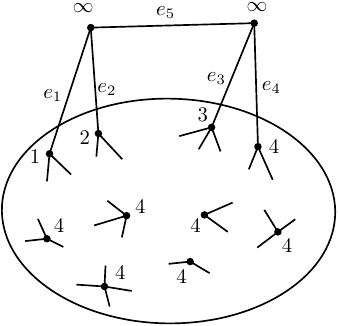}
\par
(e)
\end{minipage}
\begin{minipage}{0.32\textwidth}
\centering
\includegraphics[scale=0.55]{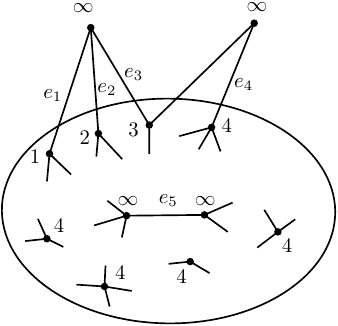}
\par
(f)
\end{minipage}
\begin{minipage}{0.32\textwidth}
\centering
\includegraphics[scale=0.55]{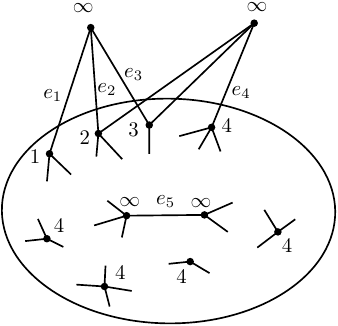}
\par
(g)
\end{minipage}
\begin{minipage}{0.32\textwidth}
\centering
\includegraphics[scale=0.55]{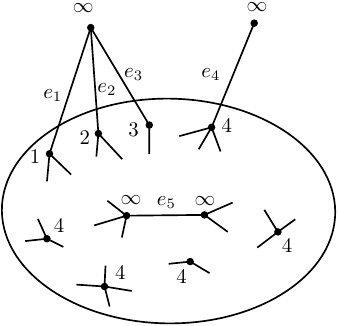}
\par
(h)
\end{minipage}
\caption{Examples of valid labellings for $t=5$.}
\label{fig:labelling}
\end{figure}

Henceforth in this section, we will simplify notation by denoting
the symmetric $2$-tensor of two points $p_1, p_2 \in \mathbb{R}^t$ by $p_1p_2$.
Note that $p_1p_2=p_2p_1$.
Let $\be_1,\dots, \be_t$ be the standard basis of $\mathbb{R}^t$.
Then $\{\be_i\be_j: 1\leq i\leq j\leq t\}$ is a basis for the vector space of all symmetric 2-tensors 
of pairs of points in $\mathbb{R}^t$.

\begin{lemma}\label{lem:operations}
Let $G=(V,E)$ be a graph with $n\geq t+1$ vertices for some positive integer $t$,
and $\ell$ be a valid labelling of $G$ with respect to $t$ distinct edges $e_1,\dots, e_t\in E$.
If $E\sm \{e_1,\dots, e_t\}$ is a base in ${\cal S}_{t-1}(K_n)$,
then $E$ is a base in ${\cal S}_t(K_n)$.
\end{lemma}
\begin{proof}
Put $G-\{e_{i+1},\dots, e_t\}=G_i=(V,E_i)$ for $i=0,\dots, t$ (where $G_t=G$).
Since $E_0$ is a base in ${\cal S}_{t-1}(K_n)$,
there is a $(t-1)$-dimensional point configuration $p:V\rightarrow \mathbb{R}^{t-1}$ 
such that $\{p(u)p(v): uv\in E_0\}$ forms 
a base of the space of symmetric $2$-tensors over $\mathbb{R}^{t-1}$.

Let $p_0:V\rightarrow \mathbb{R}^t$ be a $t$-dimensional pointconfiguration defined by $p_0(v)=\begin{pmatrix} p(v) \\ 0\end{pmatrix}$ for $v\in V$.
Then 
\begin{equation}\label{eq:operation1}
{\rm span}\{p_0(u)p_0(v): uv\in E(G_0)\}={\rm span}\{\be_a\be_b: 1\leq a\leq b\leq t-1\}.
\end{equation}

Denote $e_i=u_iv_i$.
By (i) and (ii), we have that $i=\ell(u_i)<\ell(v_i)=\infty$ for $i=1,\dots, t-1$ and $\ell(u_t)=\infty=\ell(v_t)$.
By a non-singular linear transformation,
we may assume that $p_0(u_i)=\be_i$ for $i=1,\dots, t-1$.
For $k=0,1\dots, t$, we will inductively construct a configuration $p_k:V\to \R^t$ so that $(G_k,p_k)$ has the following two properties.
\begin{description}
\item[(P1)] $p_k(w)=p_0(w)$ for all $w\in V$ with $\ell(w)\neq \infty$;
\item[(P2)] 
$\{p_k(u)p_k(v): uv\in E(G_k)\}$ is a basis for ${\rm span}\{\{\be_a\be_b: 1\leq a\leq b\leq t-1\}\cup \{\be_{a}\be_t: 1\leq a\leq k\}\}$.
\end{description}
The base case for the construction is $(G_0,p_0)$, which satisfies  properties (P1) and  (P2) by (\ref{eq:operation1}).

Suppose $(G_{k-1},p_{k-1})$ satisfies (P1) and (P2).
For any $\lambda\in \mathbb{R}$, define $p_{k,\lambda}:V\to \R^t$ by
\begin{equation}
p_{k,\lambda}(w)=\begin{cases}
    p_{k-1}(w)+\lambda \be_t & (\text{if $w$ is an endvertex of $e_k$ with $\ell(w)=\infty$)} \\ 
    p_{k-1}(w) & \text{(otherwise)}
\end{cases}
\end{equation}
for $w\in V(G)$.
Then $p_{k,\lambda}$ satisfies (P1) by definition.
We will show (P2) holds when $\lambda$ is sufficiently small. 
Let $A_{k-1}(\lambda)=\{p_{k,\lambda}(u)p_{k,\lambda}(v): uv\in E(G_{k-1})\}$.
Since $(G_{k-1},p_{k-1})$ satisfies (P2),   
$A_{k-1}(0)$ is a basis of 
${\rm span}\{\{\be_a\be_b: 1\leq a\leq b\leq t-1\}\cup \{\be_{a}\be_t: a\leq k-1\}\}$.
In particular  $A_{k-1}(0)$ is linearly independent and hence $A_{k-1}(\lambda)$ will  still be linearly independent for any sufficiently small $\lambda$. Choose a fixed $\lambda>0$ such that
$A_{k-1}(\lambda)$ is linearly independent. 
\begin{claim}\label{claim:operation}
$A_{k-1}(\lambda)$ is a basis of ${\rm span}\{\{\be_a\be_b: 1\leq a\leq b\leq t-1\}\cup \{\be_{a}\be_t: a\leq k-1\}\}$.
\end{claim}
\begin{proof}
The proof is split into two cases.

We first consider the case when $k\leq t-1$.
By property (iv) of the valid labelling $\ell$ and the assumption that $k\leq t-1$,  $\ell(w)<k\leq t-1$ for any $w\in N_{G_{k-1}}(v_k)$ (where $v_k$ is  the endvertex of $e_k$ with $\ell(v_k)=\infty$).
Hence, by properties (i), (ii) and (iii), for each $w\in N_{G_{k-1}(v_k)}$, we have $w=u_i$ for some $i\leq t-2$.
In particular, by (P1), $p_{k,\lambda}(w)=\be_{\ell(w)}$ with $\ell(w)\leq k-1$.
Therefore,
\[
p_{k,\lambda}(w)p_{k,\lambda}(v_k)\in {\rm span}\{\{\be_a\be_b: 1\leq a\leq b\leq t-1\}\cup \{\be_{a}\be_t: a\leq k-1\}\}
\]
for all $w\in N_{G_{k-1}}(v_k)$. Since $p_{k,\lambda}(w)=p_{k-1}(w)$ for all other vertices $w$, 
we conclude that, for any edge $uv$ in $G_{k-1}$,
$p_{k,\lambda}(u)p_{k,\lambda}(v)$ still belongs to 
${\rm span}\{\{\be_a\be_b: 1\leq a\leq b\leq t-1\}\cup \{\be_{a}\be_t: a\leq k-1\}\}$.
The claim now follows from 
the independence of $A_{k-1}(\lambda)$ and  dimension counting.

The case when $k=t$ is essentially identical.
The only difference  is that, for each $w\in N_{G_k}(v_k)$, we use (iii) and (P1) to represent  $p_{k,\lambda}(w)$ in the form
$p_{k,\lambda}(w)=\sum_{i=1}^{t-1}\mu_{w,i} \be_{i}$. 
\end{proof}

Note that $G_k$ is obtained from $G_{k-1}$ by adding $e_k=u_kv_k$.
Suppose $k\leq t-1$.
By (i) and (P1), $p_{k,\lambda}(u_k)=\be_k$
and $p_{k,\lambda}(v_k)=p_{k-1}(v_k)+\lambda \be_t$.
Since $\lambda\neq 0$, 
$p_{k,\lambda}(v_k)$ can be written as $p_{k,\lambda}(v_k)=\sum_{i=1}^t \mu_i \be_i$
for some scalars $\mu_i$ with $\mu_t\neq 0$.
Then,  $p_{k,\lambda}(u_k)p_{k,\lambda}(v_k)=\sum_{i=1}^t \mu_i \be_k \be_i$.
By Claim~\ref{claim:operation}, $\be_k\be_i$ is spanned by 
$A_{k-1}(\lambda)$
for all $i=1,\dots, t-1$.
Thus, by $\mu_t\neq 0$, $\be_k \be_t$ belongs to 
${\rm span}\{A_{k-1}(\lambda)\cup\{p_{k,\lambda}(u_k)p_{k,\lambda}(v_k)\}\}$, and (P2) holds for  $(G_k,p_{k,\lambda})$.

On the other hand, if $k=t$, then by (ii) and the construction,
$p_{k,\lambda}(u_t)=p_{k-1}(u_t)+\lambda\be_t$
and $p_{k,\lambda}(v_t)=p_{k-1}(v_t)+\lambda\be_t$.
Then $p_{k,\lambda}(u_t)p_{k,\lambda}(v_t)=
p_{k-1}(u_t)p_{k-1}(v_t)+p_{k-1}(u_t)\lambda\be_t+p_{k-1}(v_t)\lambda\be_t+\lambda^2\be_t\be_t$.
By Claim~\ref{claim:operation}, the first three terms belong to $A_{k-1}(\lambda)$,
and hence the fourth term $\be_t \be_t$ belongs to ${\rm span}\{A_{k-1}(\lambda)\cup\{p_{k,\lambda}(u_k)p_{k,\lambda}(v_k)\}\}$. Thus,  (P2) holds for  $(G_t,p_t)$.

Finally, by taking $k=t$, (P2)  implies the lemma.
\end{proof}

{We say that a graph $G=(V,E)$ is {\em $\MS_t$-independent} if $\rank \MS_t(G)=|E|$.}
We can use  Lemma \ref{lem:operations} to determine a large family of {\em $\MS_t$-independent graphs.

\begin{lemma}\label{lem:max_deg_two}
Suppose $t$ is a positive integer and $G=(V,E)$ is a graph with 
at most ${t+1}\choose2$ edges and 
maximum degree at most two. 
Then $G$ is  ${\cal S}_t$-independent.
    \end{lemma}
        
\begin{proof} We proceed by induction on $t$. It is straightforward to check the lemma holds when $t=1$ so we may assume that $t\geq 2$. {We may also assume without loss of generality that $G$ has ${t+1}\choose2$ edges (by adding disjoint copies of $K_2$ to $G$ if necessary) and has no isolated vertices.} Then each connected component of $G$ is either a cycle or a path of length at least one. Let $H_1,H_2,\ldots,H_m$ be the connected components of $G$ and put $|V(H_i)|=n_i$ for $1\leq i\leq m$. We may assume that $n_i\geq 3$ for $1\leq i\leq k$ and $n_i=2$ for $k+1\leq i\leq m$.

We next construct a valid vertex labelling of $G$.  For all $1\leq i\leq k$, choose a subgraph $H_i'$ of $H_i$ consisting  of $m_i=\lfloor 2n_i/3 \rfloor$ vertex disjoint paths of length two. Order the edges of $H'=\bigcup_{i=1}^k H_i'$ as $e_1,e_2,\ldots e_{2a}$ in such a way that $e_{i}=u_{i}v_{i}$ and $u_{2i-1}v_{2i-1}=u_{2i}v_{2i}$ is a path in $H$ for all $1\leq i\leq a$. If $2a< t$ then we can extend $e_1,e_2,\ldots e_{2a}$ to a sequence $e_1,e_2,\ldots e_t$ by adding edges $e_i=u_iv_i$ arbitrarily from $\bigcup_{i=k+1}^m H_i'$, since $|E|={{t+1}\choose2}$. In both cases, we can construct a valid labelling with respect to $e_1,e_2,\ldots e_t$ by putting $\ell(u_i)=i,\ell(v_i)=\infty$ for $1\leq i \leq t-1$ odd, $\ell(u_i)=\infty,\ell(v_i)=i$ for $1\leq i \leq t-1$ even,
$\ell(u_t)=\infty=\ell(v_t)$ and $\ell(v)=t-1$ for all other vertices of $G$. Then $G-\{e_1,e_2,\ldots,e_t\}$
is a base of ${\cal S}_{t-1}(K_n)$ by induction. We can now apply Lemma \ref{lem:operations} to deduce that $G$ is a base of ${\cal S}_t(K_n)$.
%
%
%
\end{proof}

\subsection{Symmetric-tensor matroids of dimension at most five}
\label{sec:sym}
{\rm Jor{\'a}n \cite{J} and 
Grasegger, Guler, Jackson, and Nixon \cite{GGJN} characterised the generic $d$-dimensional rigidity matroid ${\cal R}_d(K_n)$ for $n\leq d+6$.
By duality, this corresponds to 
characterising  ${\cal S}_t(K_n)$ for $1\leq t\leq 5$.
We will show in Theorem \ref{thm:characterization} below that ${\cal S}_t$-independence is equivalent to ${\cal C}_{t}$-independence 
for these values of $t$.
Our proof is simpler than that given in \cite{GGJN} in the sense that we avoid the use of a computer to verify the statement for small graphs.
We will need the following enumeration of the graphs in $\MC_{n,t}$ and the values of $c_t$ for $1\leq t\leq 5$.}
\begin{align*}
{\cal C}_{n,2}&=\{K_n,\bar{K}_n,K_1+\bar{K}_{n-1}\}, 
\\
{\cal C}_{n,3}&=\{K_n,\bar{K}_n,K_1+\bar{K}_{n-1}, K_2+\bar{K}_{n-2}\},
\\
{\cal C}_{n,4}&=\{K_n,\bar{K}_n,K_1+\bar{K}_{n-1}, K_2+\bar{K}_{n-2}, K_3+\bar{K}_{n-3}, K_{3,3}\dot\cup \bar K_{n-6}\},
\\
{\cal C}_{n,5}&=\{K_n,\bar{K}_n,K_1+\bar{K}_{n-1}, K_2+\bar{K}_{n-2}, K_3+\bar{K}_{n-3}, K_4+\bar{K}_{n-4}\} \\
&\hspace{2em}\cup\{K_{3,4}\dot\cup \bar K_{n-7}, K_{4,4}\dot\cup \bar K_{n-8}, (K_{3,3}\dot\cup \bar{K}_{n-7})+K_1\},
\end{align*}
and 
\begin{align*}
&\mbox{ $c_{2}(K_n)=3$; $c_{2}(K_1+\bar{K}_{n-1})=2$}; \\
&\mbox{ $c_{3}(K_n)=6$, $c_{3}(K_1+\bar{K}_{n-1})=3$, $c_{3}(K_2+\bar{K}_{n-2})=5$};\\
&\mbox{ $c_{4}(K_n)=10$, $c_{4}(K_1+\bar{K}_{n-1})=4$, $c_{4}(K_2+\bar{K}_{n-2})=7$, $c_{4}(K_3+\bar{K}_{n-3})=9$,}\\
&\mbox{ $c_{4}(K_{3,3}\dot\cup \bar K_{n-6})=8$}; \\
&\mbox{ $c_{5}(K_n)=15$, 
  $c_{5}(K_1+\bar{K}_{n-1})=5$, $c_{5}(K_2+\bar{K}_{n-2})=9$, $c_{5}(K_3+\bar{K}_{n-3})=12$,}\\&\mbox{ $c_{5}(K_4+\bar{K}_{n-4})=14$, $c_{5}(K_{3,4}\dot\cup \bar K_{n-7})=11$,
  $c_{5}(K_{4,4}\dot\cup \bar K_{n-8})=13$,} \\
&\mbox{ $c_{5}((K_{3,3}\dot\cup \bar{K}_{n-7})+K_1)=13$.}
\end{align*}

{
\begin{theorem}\label{thm:characterization}
Let $G$ be a graph and $t$ be an integer with $1\leq t\leq 5$.
Then $G$  is ${\cal S}_{t}$-independent  if and only if it is ${\cal C}_{t}$-independent.
\end{theorem}
}
\begin{proof}
{Let $G=(V,E)$ and put $|V|=n$. The theorem holds when $n\leq t+1$ since in this case $G$ is both ${\cal S}_{t}$-independent  and ${\cal C}_{t}$-independent. Hence we may assume that $n\geq t+2$.}
Since $\MS_t(K_n)$ is an abstract symmetric $t$-tensor matroid, necessity now follows from Lemma \ref{lem:necind}.

To prove sufficiency we suppose, for a contradiction, that $G$ is a ${\cal C}_{t}$-independent graph with $n$ vertices
which is not independent in ${\cal S}_{t}(K_n)$, and that $G$ has been chosen such that  $t$ is as small as possible. 
{
We may assume without loss of generality that $|E|={{t+1}\choose2}$, since adding disjoint copies of $K_2$ to $G$ will not destroy the $\MC_t$-independence of $G$ by Lemma \ref{lem:mindep}. }
Throughout the proof, we will denote the rank function of $\MS_t(K_n)$ by $r_t$ and use the fact that $r_t(F)=c_t(F)$ for all $F\in \MC_{n,t}$ by Lemmas \ref{lem:cyclic_flats} and \ref{lem:abconing}.

Since $G$ is ${\cal C}_{t}$-independent, $
K_{1}+ \bar K_{n-1}\in {\cal C}_{n,t}$ and $r_t(K_{1}+ \bar K_{n-1})=t$, we have $|E(H)|\leq t$ for all $H\subseteq G$ with $H\subseteq K_{1}+ \bar K_{n-1}$ and hence $G$ has 
maximum degree at most $t$.  

Suppose $G$ has a vertex $v$ of degree $t$.
Let $e_1,\dots, e_t$ be the edges of $G$ incident to $v$.
Then $G$ has a valid labelling with respect to $e_1,\dots, e_t$ as in Figure \ref{fig:labelling}(a). If $G-v$ is not ${\cal C}_{t-1}$-independent then there would exist $H\subseteq G-v$ and $F\in C_{n-1,t-1}$ with 
$|E(H)|>r_{t-1}(F)$. We would then have $H'=H\cup K_{1,t}\subseteq G$, $F'=F+K_1\in C_{n,t}$, $H'\subseteq F'$  and $|E(H')|=|E(H)|+t>r_{t-1}(F)+t=r_t(F')$, contradicting the hypothesis that $G$ is ${\cal C}_{t}$-independent.
Hence, $G-v$ is ${\cal C}_{t-1}$-independent, and thus is independent in ${\cal S}_{t-1}(K_{n-1})$ by induction.
We can now apply Lemma~\ref{lem:operations} to deduce that $G$ is independent in ${\cal S}_t(K_n)$, which is a contradiction.
Hence $G$ has maximum degree at most $t-1$.
Lemma \ref{lem:max_deg_two} now implies that $t\geq 4$.


\begin{claim}\label{claim:t-1}
$G$ has no vertex of degree $t-1$.
\end{claim}
\begin{proof}
Suppose, for a contradiction, that $G$ has a vertex of degree $t-1$.
In view of Lemma~\ref{lem:operations}, it will suffice to show that 
$G-(N_G(v)\cup \{v\})$ has  an edge $e^*$ such that
$G-v-e^*$ is a base of ${\cal S}_{t-1}(K_{n-1})$,
since then $G$ admits a valid labelling as shown in Figure \ref{fig:labelling}(b),  and we can apply Lemma~\ref{lem:operations} to $G-v$ 
to deduce that $G$ is not a counterexample. 




To find such an edge $e^*$, we first show that 
\begin{equation}\label{eq:claim:1}
\text{$G-v$ spans ${\cal S}_{t-1}(K_{n-1})$.}
\end{equation}

Suppose for a contradiction that (\ref{eq:claim:1}) is false. Then $|E(G-v)|={t\choose2}+1\geq r_{t-1}(G-v)+2$ so $G-v$ contains two distinct $\MS_{t-1}$-circuits $C_1,C_2$. Let $F_i$ be the closure of $C_i$ in ${\cal S}_{t-1}(K_{n-1})$ for $i=1,2$. {The choice of $t$ implies that $\MS_{t-1}$-independence and $\MC_{t-1}$-independence are equivalent and hence  $F_i\in \MC_{t-1,n-1}$ by Lemma \ref{lem:suffind}.}
 We have $F_i\neq K_{n-1}$ since 
$G-v$ does not span ${\cal S}_{t-1}(K_{n-1})$ and $F_i\neq K_{1}+\bar K_{n-2}$ since $G-v$ has maximum degree at most $t$. If $F_i=K_{t-2}+\bar K_{n-t+1}$ and $G-v\not\subseteq F_i$ then we would have $r_{t-1}(G-v)>r_{t-1}(K_{t-2}+\bar K_{n-t+1})={t\choose2}-1$ again contradicting the assumption that $G-v$ does not span ${\cal S}_{t-1}(K_{n-1})$. On the other hand, if $F_i=K_{t-2}+\bar K_{n-t+1}$ and $G-v\subseteq F_i$ then we would have $G\subseteq K_{t-1}+\bar K_{n-t+1}$ and this would contradict the $\MC_{t}$-independence of $G$. The only remaining alternative is that
$t=5$ and $F_i\in \{K_2+\bar K_{n-3}, K_{3,3}\dot\cup \bar K_{n-7}\}$ for both $i=1,2$. We cannot have $F_i=K_{3,3}\dot\cup \bar K_{n-7}$ for both $i=1,2$ since this would imply that $C_1\cong K_{3,3}\cong C_2$ and contradict the fact that no graph with ${t\choose 2}+1=11$ edges can contain two distinct copies of $K_{3,3}$. Hence we may assume that $F_1=K_2+\bar K_{n-3}$. Then $r_4(F_1)=7=r_4(C_1)$ and $C_1$ is the edge-disjoint union of two copies of $K_{1,4}$ by $|C|=r_4(C_1)+1=8$. Since $G-v$ has 11 edges, $C_1\cap C_2\neq \emptyset$ and hence $G-v$ contains a third $\MS_{t-1}$-circuit. Since $G-v$ cannot contain two distinct $K_{3,3}$'s, we may assume that two of these circuits, say $C_1,C_2$ are edge-disjoint unions of two copies of $K_{1,4}$. Since $G-v$ has 11 edges and maximum degree four, this implies that $G-v=C_1\cup C_2\subseteq K_3+\bar K_{n-4}$. This in turn gives $G\subseteq K_4+\bar K_{n-4}$ and contradicts the $\MC_{5}$-independence of $G$.
Thus,  (\ref{eq:claim:1}) holds.

By (\ref{eq:claim:1}) and $|E(G-v)|={t\choose 2}+1$,
$G-v$ contains a unique ${\cal S}_{t-1}$-circuit $C$. Let $F$ be the closure of $C$ in ${\cal S}_{t-1}(K_{n-1})$. Then $F\in {\cal C}_{n-1,t-1}$ by the minimality of $t$  and Lemma \ref{lem:suffind}.
It remains to show that $C$ contains an edge $e^*$ whose endvertices are not in  $N_G(v)$. 
We will consider the various alternatives for $F$.
\\[1mm]
Case 1: $F=K_{n-1}$. If every edge of $G-v$ is incident to $N_G(v)$, then  $G\subseteq K_{t-1}+\bar K_{n-t+1}$. This  contradicts the ${\cal C}_{t}$-independence of $G$. Thus we can choose an edge  $e^*$  of $G-v-N_G(v)$. 
\\[1mm]
Case 2: $F=K_1+\bar{K}_{n-2}$. This case cannot occur since it would imply that $C\cong K_{1,t+1}$ and contradict the fact that $G$ has maximum degree $t$.
\\[1mm]
Case 3: $F=K_2+\bar{K}_{n-3}$.
Then $|C|=r_{t-1}(K_2+\bar{K}_{n-3})+1=2(t-1)$. Since $G$ has maximum degree $t-1$, $C$ is the edge disjoint union of two copies of $K_{1,t-1}$ and
the two  vertices of $C$ of degree $(t-1)$ do not belonging to $N_G(v)$. If every edge of $C$ is incident to $N_G(v)$, then $C\cong K_{2,t-1}$ and  $N_G(v)$ is the $(t-1)$-set in the bipartition of $C$. Then $G$ would contains a copy of  $K_{3,t-1}$, contradicting the ${\cal C}_{t}$-independence of $G$. Hence some edge $e^*$ of $C$ is not incident to $N_G(v)$.
\\[1mm]
Case 4: $F=K_{3,3}\dot\cup\bar{K}_{n-7}$.
Then $C\cong K_{3,3}$ and $t=5$.
If every edge of $C$ is incident to $N_G(v)$ then we would have  $K_{3,4}\subseteq G$. This  contradicts the ${\cal C}_{5}$-independence of $G$ and hence some edge $e^*$ of $C$ is not incident to $N_G(v)$.
\\[1mm]
Case 5: $F=K_3+\bar{K}_{n-4}$.  Then $t=5$ and $|C|=r_4(K_3+\bar{K}_{n-4})+1=10$. Let $x,y,z$ be the vertices in $F$ which induce a $K_3$ ordered so that $d_C(x)\geq d_C(y)\geq d_C(z)$, and let $s$ be the number of edges of $C$ induced by $\{x,y,z\}$. Then $10+s=d_C(x)+d_C(y)+d_C(z)$.
Since each vertex of $C$ has degree at most four, we have 
$(s, d_C(x), d_C(y), d_C(z))=(0,4,3,3), (0,4,4,2), (1,4,4,3)$, or $(2,4,4,4)$. 

Suppose, for a contradiction, that every edge of $C$ is incident with $N_G(v)$.
Then the facts that $G$ has maximum degree four and 
no $K_{3,4}$-subgraph, imply that no two vertices of $C$ have degree four in $C$ and hence $(s, d_C(x), d_C(y), d_C(z))=(0,4,3,3)$.
This in turn implies that $(\{x,y,z,v\},N_G(v))$ induces a subgraph of $K_{4,4}$ with 14 edges in $G$, contradicting the $\MC_{5}$-independence of $G$. Hence some edge $e^*$ of $C$ is not incident to $N_G(v)$.

\medskip

We have shown that the  desired edge $e^*$ exists in all cases. The proof of Claim \ref{claim:t-1} is now complete. 
\end{proof}

 We can now use Lemma \ref{lem:max_deg_two} and the hypothesis that $t\leq 5$ to deduce that $t=5$ and the maximum degree of $G$ is equal to three.

Figure~\ref{fig:labelling}(c)-(h),
show how we can find a valid labelling for $G$ using vertices of degree  three. Specifically, given a vertex $v_1$ of degree three in $G$, we can find a valid labelling if either:
\begin{itemize}
\item there is a non-adjacent vertex $v_2$ of degree three 
such that $|N_G(v_1)\cap N_G(v_2)|\in \{1,2\}$, see Figure~\ref{fig:labelling}(c)(g);
\item there is a non-adjacent vertex $v_2$ of degree two 
such that $|N_G(v_1)\cap N_G(v_2)|\in \{0,1\}$, see Figure~\ref{fig:labelling}(d)(f);
\item there is an adjacent  vertex $v_2$ of degree three
such that $|N_G(v_1)\cap N_G(v_2)|=0$, see Figure~\ref{fig:labelling}(e);
\item there is a non-adjacent vertex $v_2$
of degree one such that $|N_G(v_1)\cap N_G(v_2)|=0$. Figure~\ref{fig:labelling}(h).
\end{itemize}
Note that we can always find an edge $e_5$ of $G$ which is not incident with a vertex of $N_G(v_1,v_2)\cup \{v_1,v_2\}$ in Figure~\ref{fig:labelling}(f,g,h) since $|E|=15$, $G$ is $C_{5}$-independent and $K_{4}+\bar{K}_{n-4}\in C_{n,5}$.

We will say that a pair  of vertices $v_1, v_2$ of $G$ is {\em reducible} if 
it satisfies one of the above four conditions,
and refer to  $G-\{e_1,\dots, e_5\}$ as the {\em reduced graph}.
\begin{claim}\label{claim:reduction}
Suppose $G$ has a reducible pair of vertices $v_1, v_2$,
and let $G'=G-\{e_1,\dots, e_5\}$ be the corresponding reduced graph.
If $G'$ is not ${\cal C}_{4}$-independent, then 
$G'$ has a connected component isomorphic to  $K_{3,3}$.
\end{claim}
\begin{proof}
Suppose $G'$ is not ${\cal C}_{4}$-independent.
Then $G'$ has a subgraph $H$
with $|V(H)=n\geq 5$ and $|E(H)|=r_{4}(F)+1$ for some $F\in {\cal C}_{4,n}$ with $H\subseteq F$.
Observe that, if $F\in \{K_a+\bar{K}_{n-a}: a=1,2,3\}$,
then $H$ would have a vertex of degree at least four, since $n\geq 5$.
This would contradict the fact that $G$ has maximum degree three, so
$F$ must be  a copy of $K_{3,3}\dot\cup \bar K_{n-6}$, and $E(H)=E(K_{3,3})$ follows by $r_4(K_{3,3})=8$.
Hence $G'$ contains a subgraph isomorphic to $K_{3,3}$.
Since the maximum degree of $G'$ is three, 
this subgraph isomorphic to $K_{3,3}$ must be a connected component of $G'$.
\end{proof}

Suppose $G$ has a reducible pair of vertices $v_1, v_2$.
If the corresponding reduced graph $G':=G-\{e_1,\dots e_5\}$ is ${\cal C}_{4}$-independent,
then $G'$ is a base of  ${\cal S}_4(K_n)$ 
and $G$ is a base of  ${\cal S}_5(K_n)$ by Lemma~\ref{lem:operations}.
Hence, we may assume that $G'$ is not ${\cal C}_{4}$-independent.
Then, by Claim~\ref{claim:reduction}, $G'$ has a connected component $H$ isomorphic to $K_{3,3}$. Since $G$ has maximum degree three, $H$ remains as a connected component in $G$. 
Since $|E(G)|= 15$, $G$ consists of $e_1,\dots, e_5$, $H$, and one more edge $f$.
Then any two adjacent vertices of $H$
form a reducible pair in $G$ (of Type (e) in Figure~\ref{fig:labelling}).
Since  the corresponding reduced graph contains no copy of $K_{3,3}$,
Claim~\ref{claim:reduction} and Lemma~\ref{lem:operations} imply that $G$ is a base of ${\cal S}_5(K_n)$, contradicting the choice of $G$.
Hence,  
\begin{equation}\label{clm:redcue}
\mbox{$G$ has no reducible pairs of vertices.}    
\end{equation}
 This implies that one of the following two alternatives holds for all
pairs of distinct vertices $v_1, v_2\in V$ in which $v_1$ has degree three.
\begin{description}
\item[(A)] If $v_1, v_2$ both have degree three, then either
$v_1$ and $v_2$ are adjacent and $|N_G(v_1)\cap N_G(v_2)|\in \{1,2\}$, or
$v_1$ and $v_2$ are not adjacent and
$|N_G(v_1)\cap N_G(v_2)|\in \{0,3\}$.
\item[(B)] If $v_1$ has degree three and $v_2$ has degree less than three, then  either $v_1$ and $v_2$ are adjacent or $N_G(v_2)\subset  N_G(v_1)$.
\end{description}

We first show that: 
\begin{equation}\label{eq:claim1}
\text{there is no pair $v_1, v_2\in V$, both of degree three, satisfying $N_G(v_1)\cap N_G(v_2)=\emptyset$.}
\end{equation}
To see this, suppose there is such a pair $v_1, v_2$.
Then $v_1$ and $v_2$ are not adjacent by (A).
We claim that $G$ is 3-regular.
Indeed, if $w$ is a vertex of degree less than three,
then by (B), $w\in N(v_i)$ or $N(w)\subset N(v_i)$ holds for each $i=1,2$.
By $N(v_1)\cap N(v_2)=\emptyset$, we may assume
$w\in N(v_1)$ and $N(w)\subset N(v_2)$.
This however implies $v_1\in N(w)\subset N(v_2)$, contradicting that $v_1$ and $v_2$ are not adjacent.
Thus, $G$ is 3-regular. It is straightforward to check that the only cubic graph with the property that all edges belong to a triangle is the disjoint union of copies of $K_4$. This and the hypothesis that $|E|=15$ imply that $G$ has two adjacent vertices with no common neighbours.
This contradicts the fact that $G$ has no reducible pair,
and  (\ref{eq:claim1}) follows.

We next show that: 
\begin{equation}\label{eq:claim2}
\text{$G$ is connected.}
\end{equation}
Suppose, for a contradiction, that $G$ is not connected. Then we can choose two vertices $v_1, v_2$ belonging to distinct components of $G$.
By Lemma~\ref{lem:max_deg_two}, we may assume that $v_1$ has degree three.
By (A) and (B), $v_2$ has degree three and $N(v_1)\cap N(v_2)=\emptyset$.
This contradicts (\ref{eq:claim1}).

We next show that: 
\begin{equation}\label{eq:claim3}
\text{all pairs of degree three vertices of $G$ are  adjacent.}
\end{equation}
Suppose that two vertices $v_1, v_2$ of degree three are not adjacent.
By (A) and (\ref{eq:claim1}), we have
$N(v_1)=N(v_2)=\{x,y,z\}$, say.
Since $G$ is a connected graph with 15 edges and maximum degree three, at least one of the neighbours of $v_1,v_2$, say $x$, has a neighbour in $V\sm\{v_1,v_2,x,y,z\}$. 
Then $v_1,x$ is a reducible pair (as shown in Figure~\ref{fig:labelling}(e)), contradicting (\ref{clm:redcue}),

Finally, one can easily check that 
a connected graph of maximum degree three which satisfies (\ref{eq:claim3}) and (B) has at most 9 edges. This contradiction completes the proof of the theorem.
\end{proof}

{\rm
It is possible that Theorem \ref{thm:characterization} can be extended to cover the case when $t=6$. However it becomes false when $t=7$. This follows since $K_{6,6}$ is a circuit in $\MR_4(K_{12})$ and the graph obtained by adding two edges $e,f$ to opposite sides of $K_{6,6}$ is a hyperplane. By duality, this implies that the graph obtained by deleting $e,f$ from  different components of $K_6\dot\cup K_6$ is a circuit in $\MS_7(K_{12})$, so has rank 27. On the other hand, the only hyperplane of $\MS_7(K_{12})$ contained in $\MC_{12,7}$ is $K_6+\bar K_6$. Since $(K_6-e)\dot\cup (K_6-f)\not\subseteq K_6+\bar K_6$, we cannot use $\MC_{7}$-independence to deduce that $(K_6-e)\dot\cup (K_6-f)$ is dependent in $\MS_7(K_{12})$.

\appendix
\addcontentsline{toc}{section}{Appendix}

\section{Appendix}
\label{sec:K1t}
We show that the matroid $M_1$ defined in the proof of Theorem \ref{thm:K1t} is the unique maximal $K_{1,t+1}$-matroid on $K_{n}$ when $t= 2,3$ and characterise $M_1$ for all $t$ when $n\geq 2\lfloor (2t+1)/3\rfloor$.

\subsection{Matroid Preliminaries}
We first give some concepts and results on maximal matroids from \cite{JT}.
Let $\MX$ be a family of subsets of a finite set $E$.
A matroid on $E$ is said to be an $\MX$-matroid if every $X\in \MX$ is a circuit in $M$. 
A {\em proper $\MX$-sequence} in $E$ is a sequence $\MS=(X_1, X_2, \dots, X_k)$ of sets in $\MX$ such that $X_i\not\subset \bigcup_{j=1}^{i-1} X_j$ for all $i=2,\dots, k$. For $F\subset E$, let ${\rm val}(F,\MS)=| F\cup (\bigcup_{i=1}^k X_i)|-k$ and define $\val_\MX:2^E\to \ZZ$ by $\val_\MX(F)=\min\{\val(F,\MS)\}$ where the minimum is taken over all proper $\MX$-sequences $\MS$ in $E$.

\begin{lemma}[{\cite[Lemma 3.3]{CJT2}}]\label{lem:upper1}
Suppose $M=(E,r)$ is an $\MX$-matroid and $F\subseteq E$. Then $r_{M}(F)\leq {\rm val}(F,\MS)$ for any proper $\MX$-sequence $\MS$.  Furthermore, if equality holds, then $r(F-e)=r(F)-1$ for all 
$e\in F\sm  (\bigcup_{X\in \MS} X)$ and $r(F+e)=r(F)$ for all 
$e\in \bigcup_{X\in \MS} X$. 
\end{lemma}


We can often use a special kind of proper $\MX$-sequence to verify the equality
 $r(F)={\rm val}(F,\MS)$. Given $F\subseteq E$ and $F_0\subseteq F$,
a {\em weakly $\MX$-saturated sequence for $F$ from $F_0$} is 
a proper $\MX$-sequence $\MS=(X_1, X_2,  \dots, X_m)$ such that $|X_i\setminus (F_0\cup \bigcup_{j=1}^{i-1} X_i)|=1$ for all $i$ with $1\leq i\leq m$ and 
$F=F_0\cup \bigcup_{i=1}^m X_i$. Given such a sequence $\MS$, it is easy to check that $\val(F,\MS)=|F_0|$.

The {\em truncation} of a matroid $M_1=(E,\MI_1)$ of rank $k$ is the matroid $M_0=(E,\MI_0)$ of rank $k-1$, where $\MI_0=\{I\in \MI_1\,:\,|I|\leq k-1\}$.
Crapo~\cite{C} defined {\em matroid erection} as the `inverse operation' to truncation. So $M_1$ is an {\em erection} of $M_0$ if $M_0$ is the truncation of $M_1$. (For technical reasons we also consider $M_0$ to be a {\em trivial erection} of itself.)
Note that, although every matroid  has a unique truncation,  matroids may have several, or no,  non-trivial erections.

Crapo~\cite{C} showed that the poset of all erections of a matroid $M$ is actually a {lattice}. It is clear that the trivial erection of $M$ is the unique minimal element in this lattice. Since this is a finite lattice, there also exists a unique 
maximal element $M'$ which Crapo called the  {\em free erection of $M$}.  Several authors gave algorithms for constructing the free erection of $M$ based on calculations on the  dual matroid $M^*$. A version of these algorithms which works directly with $M$ is given below. It uses the following concepts.

A set $X\subseteq E$ is a {\em cyclic set of $M$} if, for every $e\in X$, there  is a ciruit $C$ of $M$ with $e\in C\subseteq X$. A family $\MX$ of subsets of $E$ is said to be a {\em cyclic family} of $M$ if every $X\in \MX$ is cyclic in $M$. 
A cyclic family $\MX$ is {\em down closed in $M$} if $X'\in \MX$ whenever $X'$ is cyclic in $M$ and $X'\subseteq X$ for some $X\in \MX$. 
The {\em down closure} of  an arbitrary cyclic family $\MX$ is the down closed cyclic family $\MX^{\downarrow}$ given by
\[
\MX^{\downarrow}=
\bigcap\{ {\cal Y}: {\cal Y} \text{ is a down closed cyclic family of $M$ and } {\cal X}\subseteq {\cal Y}\}.
\]
A cyclic family $\MX$ is {\em modular} if it is down closed and   
\begin{itemize}
\item $X_1\cup X_2\in \MX$ whenever $X_1,X_2\in \MX$ and $r(X_1)+r(X_2)=r(X_1\cup X_2)+r(X_1\cap X_2)$. 
\end{itemize}
The {\em modular closure} of a cyclic family $\MX$ is the modular cyclic family $\bar \MX$ given by
\[
\bar{{\cal X}}=
\bigcap\{ {\cal Y}: {\cal Y} \text{ is a modular cyclic family of $M$ and } {\cal X}\subseteq {\cal Y}\}.
\]
It can be calculated recursively from $\MX$ using the following algorithm from \cite{CJT2}.

\paragraph{Algorithm 1}
\begin{itemize}
\item Initialize ${\cal X}_0:={\cal X}^{\downarrow}$.
\item Repeatedly construct ${\cal X}_i$ from ${\cal X}_{i-1}$ by
\[
{\cal X}_i:={\cal X}_{i-1}\cup \{X\cup Y \mid X, Y\in {\cal X}_{i-1}:  \text{ $X$ and $Y$ form a modular pair in $M$} \}^{\downarrow}.
\]
\item If ${\cal X}_i={\cal X}_{i-1}$, then ${\cal X}_i$ is $\bar{{\cal X}}$.
\end{itemize}

Let $\MX_0$ be the family of all non-spanning cyclic sets of $M$. Then each $X\in \MX_0$ is cyclic in every erection of $M$. More generally, it can be seen that  each $X\in \bar\MX_0$ is cyclic in every erection of $M$ and that $\bar \MX_0$ is the family of non-spanning cyclic sets in the free erection of $M$. We also have the following explicit rank formula for the free erection of $M$. Given $X\subseteq E$ we use $\cyc(X)$ to denote the largest subset of $X$ which is cyclic in $M$. (This is well defined since the union of two cyclic sets is cyclic.)
 
\begin{lemma} \cite[Corollary 2.4] {CJT2} \label{lem:freerank}
Let $M=(E,r)$ be a matroid, $\MX_0$ be the family of all non-spanning cyclic sets of $M$ and $M'=(E,r')$ be the free erection of $M$.
Then, for all $X\subseteq E$, we have 
\begin{equation}
 \label{eq:erection}
r'(X)=\begin{cases}
r(X) & \mbox{ if ${\rm cyc}(X)\in \bar\MX_0$} \\
r(X)+1 & \mbox{ if ${\rm cyc}(X)\not \in \bar\MX_0$}
\end{cases}
\end{equation}
\end{lemma}

\begin{lemma}  \label{lem:freeseq}
Let $M=(E,r)$ be a matroid, $\MX_0$ be the family of all non-spanning cyclic sets of $M$ and $\MC_0$ be the family of all non-spanning circuits of $M$. Suppose $X$ is a cyclic set in $M$ and $\val(\MS,X)=r(X)$ for some proper $\MC_0$-sequence $\MS$ in $E$. 
Then $X\in \bar\MX_0$.
\end{lemma}
\begin{proof} Let $M'=(E,r')$ be the free erection of $M$. Since $X$ has rank at most $\val(\MS,X)$ in every $\MC_0$-matroid and $r'\geq r$, we have $r'(X)=r(X)$. Lemma \ref{lem:freerank} and the fact that $\cyc(X)=X$ now give $X\in \bar\MX_0$.
\end{proof}

\medskip

A {\em partial elevation} of  $M_0$ is any matroid $M$ which can be constructed from $M_0$ by a  sequence of erections. A {\em (full) elevation} of $M_0$ is a partial elevation $M$ which has no non-trivial erection.
The {\em free elevation} of $M_0$ is the matroid we get from $M_0$ by recursively constructing a maximum sequence of 
non-trivial 
free erections. 
The set  of all partial elevations of $M_0$ forms a poset $P(M_0)$ under the weak order and $M_0$ is its unique minimal element.
Every maximal element of $P(M_0)$ will have no non-trivial erection so will be a full elevation of $M_0$. 
Given Crapo's result that the poset of all erections of $M_0$ is a lattice, it is tempting to conjecture that $P(M_0)$ will also be a lattice and that the free elevation of $M_0$ will be its unique maximal element. But this is false:  counterexamples  are given in \cite{B86,JT}. 
The following weaker result is given in  \cite[Lemma 3.1]{CJT2}.

\begin{lemma}\label{lem:free_elevation}
Let $M_0$ be a matroid on $E$.
and ${\cal C}_0$ be the family of all non-spanning circuits in $M_0$.
Then the free elevation of $M_0$ is a  maximal ${\cal C}_0$-matroid on $E$.
\end{lemma}


The same proof technique can be used to deduce the following version of this lemma for partial elevations.

\begin{lemma}\label{lem:unique_partial}
Let $M_0$ be a matroid on $E$ with $\rank M_0\leq k$
and ${\cal C}_0$ be the family of all non-spanning circuits in $M_0$.
Then the free elevation of $M_0$ to rank $k$ is a  maximal 
${\cal C}_0$-matroid on $E$
of rank at most $k$.
\end{lemma}

Our next lemma is a slight strengthening of \cite[Lemma 2.1]{JT}.

\begin{lemma}\label{lem:free_elevation1}
Let $M_0$ be a matroid with groundset $E$, $\MC_0$ be the set of non-spanning circuits of $M_0$ and $M=(E,r)$ be a partial elevation of $M_0$ to rank $k$.
Suppose that,
for every non-spanning connected flat $F$ of $M$, there is a proper $\MC_0$-sequence  $\MS$
with  $r(F)={\rm val}(F,\MS)$.
Then $r(Y)={\rm val}_{\MC_0}(Y)$ for all non-spanning $Y\subset E$, 
$M$ is the unique maximal $\MC_0$-matroid on $E$ of rank at most $k$, and $M$ is equal to the partial free elevation of $M$ to rank $k$. Furthermore, 
if ${\rm val}(E,\MS)=k$ for some proper $\MC_0$-sequence  $\MS$, then 
$r=\val_{\MC_0}$, $M$ is the unique maximal $\MC_0$-matroid on $E$  and $M$ is equal to the free elevation of $M$.
\end{lemma}

\begin{proof}
We first show that $r(Y)={\rm val}_{\MC_0}(Y)$ for all non-spanning $Y\subset E$.
Since  $r_{M}(Y)\leq \val_{\MC_0}(Y)$  by Lemma~\ref{lem:upper1}, we can verify this
by showing that, for every non-spanning $Y\subseteq E$, 
there is a proper ${\MC_0}$-sequence $\MS$  such that $r_{M}(Y)={\rm val}(Y,\MS)$. 

Suppose, for a contradiction, that this is false for some set $Y$. 
We may assume that $Y$ has been chosen such that $r_M(Y)$ is as small as possible  and, subject to this condition, $|Y|$ is as large as possible.
If $Y$ is not a flat then $r_M(Y+e)=r_M(Y)$ for some $e\in E\sm Y$ and we can now use the maximality of $|Y|$ to deduce that there exists a proper  ${\MC_0}$-sequence $\MS$  such that 
$r_{M}(Y+e)={\rm val}(Y+e,\MS)$.
By Lemma~\ref{lem:upper1} and $r_{M}(Y)=r_{M}(Y+e)$, $e\in \bigcup_{X\in \MS} X$.
Hence,  ${\rm val}(Y+e,\MS)={\rm val}(Y,\MS)=r_{M}(Y+e)=r_{M}(Y)$. This would contradict the choice of $Y$. Hence $Y$ is a flat.

Suppose $Y$ is not connected. Then we have $r_M(Y)=r_M(Y_1)+r_M(Y_2)$ for some partition $\{Y_1,Y_2\}$ of $F$. Since 
$M$ is loopless,
$Y_i$ is a flat of $M$ and  $1\leq r_M(Y_i)<r_M(Y)$ for both $i=1,2$. The choice of $Y$ now implies that there exists 
a proper ${\MC_0}$-sequence $\MS_i$  such that $r_{M}(Y_i)={\rm val}(Y_i,\MS_i)$ for $i=1,2$. Since each $Y_i$ is a flat, we have $X_i\subseteq Y_i$ for all $X_i\in \MS_i$ by  Lemma~\ref{lem:upper1}. This implies that the concatenation  $\MS=(\MS_1,\MS_2)$ is a proper ${\MC_0}$-sequence and satisfies
$${\rm val}(Y,\MS)={\rm val}(Y_1,\MS_1)+{\rm val}(Y_2,\MS_2)=r_{M}(Y_1)+r_{M}(Y_2)=r_{M}(Y).$$ 
This contradicts the choice of $Y$. 

Hence 
$Y$ is a connected flat and we can use  the hypothesis of the lemma to deduce that there is a proper ${\MC_0}$-sequence $\MS$  such that $r_{M}(Y)={\rm val}(Y,\MS)$. This again contradicts our choice of $Y$ and completes the proof  that $r(Y)={\rm val}_{\MC_0}(Y)$ for all non-spanning $Y\subset E$.

We next show that $M$ is the unique maximal $\MC_0$-matroid on $E$ of rank at most $k$.
Let $M'$ be any other $\MC_0$-matroid on $E$ of rank at most $k$. Then $r_{M'}(Y)\leq {\rm val}_{\MC_0}(Y)$ by Lemma \ref{lem:upper1}. This implies that $r_{M'}(Y)\leq r(Y)$ when $Y$  is non-spanning in $M$. The same inequality holds trivially when  $Y$ spans $M$ since $M'$ has rank at most $k$ and $M$ has rank $k$. Hence 
$M$ is the unique maximal $\MC_0$-matroid on $E$ of rank at most $k$. 

The fact that $M$ is the free elevation of $M_0$ to rank $k$ now follows from Lemma \ref{lem:unique_partial}. This completes the proof of the first part of the lemma.

To prove the second part, we suppose that ${\rm val}(E,\MS)=k$ for some  proper $\MC_0$-sequence  $\MS$. Since ${\rm val}(E,\MS)$ is an upper bound on the rank of every $\MC_0$-matroid by Lemma \ref{lem:upper1}, the first part of the lemma
now implies that $M$ is the unique maximal $\MC_0$-matroid on $E$ and $M$ is equal to the free elevation of $M_0$.
\end{proof}

We will also need a result on the family of cyclic flats of a matroid.
Brylawski \cite[Proposition 2.1]{B} showed that a matroid is completely determined by its cyclic flats and their ranks.  The following result of Sims \cite{S}, subsequently rediscovered by Bonin and de Meir \cite[Theorem 3.2]{BM}, characterises when a given family of ranked sets is the family of cyclic flats of a matroid.

\begin{theorem}\label{thm:cyclic_flats}
Given a collection $\MZ$ of subsets of some ground set $E$ and a
function $r$ mapping $\MZ$ to the nonnegative integers, there is a matroid $M$ on $E$ that has $\MZ$ as its lattice of
cyclic flats, and $r$ the ranks of those cyclic flats, if and only if\\
(Z0) $\MZ$ is a lattice under inclusion;\\
(Z1) the minimum element $0_\MZ$ in this lattice has $r (0_\MZ) = 0$;\\
(Z2) $ 0 < r (Y ) - r (X) < |Y \sm X|$ for all $X,Y\in \MZ$ with $X \subsetneq Y$;\\
(Z3) for all $X,Y \in \MZ$,
$r (X) + r (Y ) \geq r (X \wedge Y ) + r (X \vee Y ) + |(X \cap Y ) \sm (X \wedge Y )|.$
\end{theorem}

\subsection{\boldmath Maximal $K_{1,t}$-matroids on $K_n$}

We assume throughout this section that $t\geq 3$ and $n\geq t+1$ are fixed integers. We first give an upperbound on the rank of any $K_{1,t}$-matroids on $K_n$ which follows immediately from \cite[Lemma 3.6]{JT}.

\begin{lemma}\label{lem:ub}
The rank of any $K_{1,t}$-matroid on $K_n$  is at most $t\choose2$. 
\end{lemma}

Following \cite{JT}, we define the {\em uniform $K_{1,t}$-matroid on $K_n$}
to be the rank $t$ matroid on $K_n$ in which every non-spanning circuit is a copy of $K_{1,t}$.

\begin{theorem}\label{thm:S3}\cite{JT}
The uniform $K_{1,3}$-matroid on $K_n$ 
is the  unique maximal $K_{1,3}$-matroid on $K_n$ and its rank function is $\val_{K_{1,3}}$. 
\end{theorem}

This result implies, in particular, that the uniform $K_{1,3}$-matroid on $K_n$ has no nontrivial erection. We will determine the free elevation of   the uniform $K_{1,t}$-matroid  on $K_n$ for all $t\geq 3$ when $n$ is sufficiently large. Our characterisation will imply that this free elevation is the unique maximal 
$K_{1,t}$-matroid  on $K_n$ only when $t=3,4$. To do this we first define a particular matroid $M_{t,n}$ on $K_n$. 


Put $f(b)=(t-1)b-{b\choose2}$ for all integers $b\geq 0$.
Let $K_n=(V,E)$ and $\MZ=\{F\subseteq E: K_n[F]\cong K_b+\bar K_{n-b}, 1\leq b\leq \tau-1\}\cup \{\emptyset,E\}$, where $\tau=\lfloor (2t+1)/3\rfloor$. Put $r(F)=f(b)
$ 
whenever $F\in \MZ$ with $K_n[F]\cong K_b+\bar K_{n-b}$, and put $r(\emptyset)=0$ and $r(E)=f(\tau)$.
We can use Theorem \ref{thm:cyclic_flats} to show 

\begin{lemma}\label{lem:matroid}
There exists a matroid $M_{t,n}$ on $K_n$ such that $\MZ$ is the lattice of
cyclic flats in $M_{t,n}$, and $r(F)$ is the rank of $F$ in $M_{t,n}$ for all $F\in \MZ$. 
\end{lemma}
\begin{proof}
By Theorem \ref{thm:cyclic_flats}, it will suffice to show that $\MZ$ and $r$ satisfy axioms $(Z_0)-(Z_3)$.

\medskip
\noindent
\textbf{(Z0)}  Consider the poset consisting of the elements of $\MZ$ ordered by inclusion and the lattice $\ML=\{U\subseteq V:1\leq |U|\leq  \tau-1\}\cup \{\emptyset,V\}$ ordered by inclusion.  (Note that $\ML$ is a lattice since it is obtained  by truncating the lattice of all subsets of $V(K_n)$ ordered by inclusion, to height $\tau+1$. 
Define $\alpha:\MZ\to \ML$ as follows. For each $F\in \MZ\sm\{\emptyset,E\}$ with $K_n[F]\cong K_b+\bar K_{n-b}$ for some $1\leq b\leq \tau$ we put $\alpha(F)$ equal to the vertex set of the copy of $K_b$ in $K_n[F]$.  In addition we put $\alpha(\emptyset)=\emptyset$ and $\alpha(E)=V$. Then  $\alpha$ is an order preserving bijection and hence $\MZ\cong\ML$ is a lattice.

\medskip
\noindent
\textbf{(Z1)} We have $0_\MZ=\emptyset$ and $r(\emptyset)=0$.

\medskip
\noindent
\textbf{(Z2)} Choose $X_1,X_2\in \MZ$ with $X_1 \subsetneq X_2$. Let $\alpha(X_i)=U_i$  for $i=1,2$ and put $b_i=|U_i|$ when $X_i\neq E$ and $b_i= \tau$ when $X_i=E$. Then  $U_1 \subsetneq U_2$ and $0\leq b_1<b_2$,
%
%
$r(X_i)=(t-1)b_i-{b_i\choose2}$ for $i=1,2$, $|X_1|=b_1(n-b_1)+{b_1\choose2}$ and $|X_2|\geq 
b_2(n-b_2)+{b_2\choose2}$.
 This gives 
$$\mbox{$r(X_2)-r(X_1)=(t-1)(b_2-b_1)-{b_2\choose2}+{b_1\choose2}=(b_2-b_1)(t-1-\tfrac{b_2+b_1-1}{2})$}$$
 and 
$$\mbox{$
|X_2\sm X_1|\geq b_2(n-b_2)+{b_2\choose2}-b_1(n-b_1)-{b_1\choose2}=(b_2-b_1)(n-\tfrac{b_2+b_1+1}{2})$}.$$
We can now deduce that $0<r(X_2)-r(X_1)< |E\sm X_1|$ since   $b_2>b_1$,   
$b_2+b_1\leq 2\tau-1$ and $n\geq t+1$.

\medskip
\noindent
\textbf{(Z3)} Choose $X_1,X_2\in \MZ$. 

We first consider the case when $X_2=E$. We have $X_1\wedge E=X_1=X_1\cap E$, $X_1\vee E=E$ and  
$$
r (X_1) + r (E) = r (X_1 \wedge E ) + r (X_1 \vee E ) + |(X_1 \cap E ) \sm (X_1 \wedge E )|$$
holds trivially.

Hence we may suppose that $X_1\neq E\neq X_2$.
Let $\alpha(X_i)=U_i$  for $i=1,2$ and put $U_3=U_1\cap U_2$. Let $|U_i|=b_i$  for $i=1,2,3$. Then   $0\leq b_1,b_2,b_3\leq \tau-1$. We have $X_1\wedge X_2=\alpha^{-1}(U_3)$, $X_1\vee X_2=\alpha^{-1}(U_1\cup U_2)$ if $|U_1\cup U_2|\leq \tau-1$ and $X_1\vee X_2=E$ if $|U_1\cup U_2|\geq \tau$. This gives $r(X_1\wedge X_2)=f(b_3)$, $r(X_1\vee X_2)=f(b_1+b_2-b_3)$ if $|U_1\cup U_2|\leq \tau$ and $r(X_1\vee X_2)=f( \tau+1)$ if $|U_1\cup U_2|\geq \tau+1$. In addition, $(X_1 \cap X_2 ) \sm (X_1 \wedge X_2 )$ consists of all the edges in $K_n$ between $U_1\sm U_2$ and $U_2\sm U_1$ so
$|(X_1 \cap X_2 ) \sm (X_1 \wedge X_2 )|=(b_1-b_3)(b_2-b_3)$.
To verify (Z3) we will also use the following identity which can be verified by counting the number of edges in the complete graph on $U_1\cup U_2$.
\begin{equation}\label{eq:id5}
\mbox{${{b_1+b_2-b_3}\choose2}={{b_1}\choose2}+{{b_2}\choose2}-{{b_3}\choose2}+(b_1-b_3)(b_2-b_3)$.}
\end{equation}

If $|U_1\cup U_2|\leq \tau-1$, then (\ref{eq:id5}) gives
\begin{align*}
&\mbox{$r(X_1 \wedge X_2 ) + r (X_1 \vee X_2 ) + |(X_1 \cap X_2 ) \sm (X_1 \wedge X_2)|$}\\
&=(t-1)(b_1+b_2)-
\mbox{${b_3\choose2}-{b_1+b_2-b_3\choose2}$}+(b_1-b_3)(b_2-b_3)\\
&=\mbox{$(t-1)(b_1+b_2)-{b_1\choose2}-{b_2\choose2}$}\\
&=r(X_1)+r(X_2).
\end{align*}

On the other hand, if $|U_1\cup U_2|\geq \tau$, then 
\begin{align*}
&r(X_1 \wedge X_2 ) + r (X_1 \vee X_2 ) + |(X_1 \cap X_2 ) \sm (X_1 \wedge X_2)|\\
&=(t-1)b_3-
\mbox{${b_3\choose2}+(t-1)\tau-{{\tau}\choose2}$}\mbox{$+(b_1-b_3)(b_2-b_3)$.}
\end{align*}
Equation (\ref{eq:id5}) now gives
\begin{align*}
r (X_1) &+ r (X_2) - r (X_1 \wedge X_2 ) - r (X_1 \vee X_2 ) - |(X_1 \cap X_2 ) \sm (X_1 \wedge X_2 )|\\
&=(t-1)(b_1+b_2-b_3)-\mbox{${b_1+b_2-b_3\choose2}-(t-1)\tau+{{\tau}\choose2}$}\\
&=f(b_1+b_2-b_3)-f(\tau)\geq 0
\end{align*}
since $b_1+b_2-b_3\geq \tau$, 
$b_1+b_2-b_3\leq 2\tau-2$
and $f$ is concave.
\end{proof}

We will abuse notation and continue to use $r$ for the rank function of $M_{t,n}$.
Note that $M_{t,n}$ is a 
$K_{1,t}$-matroid on $K_n$ since we have $r(E(K_{1,n}))=t-1$ so, by symmetry,  $r(E(K_{1,t}))=t-1=r(E(K_{1,t-1}))$ for every copy of $K_{1,t}$ in $K_n$.

We can use the characterisation of the cyclic flats in  $M_{4,n}$ to show that it is the unique maximal $K_{1,4}$-matroid on $K_n$.

\begin{theorem}\label{thm:S4max} $M_{4,n}$ is the unique maximal  $K_{1,4}$-matroid on $K_n$ and its rank function is $\val_{K_{1,4}}$. 
\end{theorem}
\begin{proof}
Let $M_{4,n}=(E,r)$.
We have $\rank M_{4,n}=f(3)=6$ and the non-spanning connected flats of $M_{4,n}$ are the copies of $K_b+\bar K_{n-b}$ in $K_n$ with $1\leq b\leq 2$ by Lemma \ref{lem:matroid}. It is straightforward to check that, for  each non-spanning connected flat $F$ of $M_{4,n}$,  there is a proper $K_{1,4}$-sequence $\MS$ with $\val(F,\MS)=r(F)$. In addition, there exists a proper $K_{1,4}$-sequence $\MS$ with $\val(E,\MS)=r(E)=6$  since $K_n$ can be constructed from $K_4$ by a weakly saturated $K_{1,4}$-sequence and $|E(K_4)|=6$. The theorem now follows from Lemma \ref{lem:free_elevation}. 
\end{proof}

Lemma  \ref{lem:free_elevation} and  Theorem \ref{thm:S4max} imply that $M_{4,n}$ is the free elevation of the uniform $K_{1,4}$-matroid on $K_n$.
We next investigate whether this extends to $K_{1,t}$-matroids when $t\geq 5$.

\begin{lemma}\label{lem:partial}
$M_{t,n}$ is the
partial free elevation of the uniform  
$K_{1,t}$-matroid on $K_n$ to rank  $f(\tau)$ and $r(Y)=\val_{K_{1,t}}(Y)$ for all non-spanning $Y\subset E$.   
\end{lemma}
\begin{proof}
By Lemma \ref{lem:free_elevation1},  it will suffice to show that,   
for every non-spanning connected flat $F$ of $M$, there is a proper $K_{1,t}$-sequence  $\MS$ in $K_n$
with  $r(F)={\rm val}(F,\MS)$. 

Let $F$ be a  non-spanning connected flat in $M$. Then $K_n[F]\cong K_b+\bar K_{n-b}$  for some  $1\leq b\leq \tau$ and $r(F)=b(t-1)-{b\choose2}$. Choose $F_0\subseteq F$ with $K_n[F_0]\cong K_b+\bar K_{t-b+1}$. Then each vertex in the $b$-set of $K_n[F_0]$ is the central vertex of a copy of $K_{1,t}$ and these $b$ copies of $K_{1,t}$ give us a proper $K_{1,t}$-sequence $\MS_0$ with $\val(F_0,\MS_0)=b(t-1)-{b\choose2}$.
In addition,  $F$ can be constructed  from $F_0$ by a weakly saturated $K_{1,t}$-sequence $\MS_1$, and the concatenation of these two sequences $\MS=(\MS_0,\MS_1)$  satisfies $r(F)= b(t-1)-{b\choose2}=\val(F,\MS)$.
\end{proof}

Our next result shows that  $M_{t,n}$ is the free elevation of the uniform $K_{1,t}$-matroid on $K_n$ when $n$ is sufficiently large.

\begin{theorem}\label{thm:St} $M_{t,n}$  is the free elevation of the uniform $K_{1,t}$ matroid on $K_n$ whenever $n\geq 2\lfloor (2t+1)/3\rfloor$. 
\end{theorem}
\begin{proof}
Put $M_{t,n}=M=(E,r)$ and $\tau=\lfloor (2t+1)/3\rfloor$. 
Since $M$ is the partial free elevation of  the uniform $K_{1,t}$ matroid on $K_n$ to rank $f(\tau)$ by Lemma \ref{lem:partial}, it will suffice to show that the free erection $M'=(E,r')$ of $M$ is equal to $M$. 

Let 
$\MX_0$ be the family of nonspanning cyclic sets of  $M$ and $\bar \MX_0$ the modular closure of $\MX_0$ in $M$. 
We first show that
\begin{equation}\label{eq:St}
\mbox{$X\in \bar \MX_0$ whenever $X\subseteq E$ and $K_n[X]\cong K_\tau+\bar K_{t+1-\tau}$}.
\end{equation}
To see this note  that, since $M$ is a $K_{1,t}$-matroid, 
$K_n[\cl(X)]$ contains a copy of $K_{\tau}+\bar K_{n-\tau}$. Lemma \ref{lem:matroid} now implies that $\cl(X)=E$ so  $X$ spans $M$. In addition, since $K_n[X]\cong K_\tau+\bar K_{t+1-\tau}$, we can find a proper sequence $\MS=(X_1,X_2,\ldots,X_\tau)$ in $E$ such that $X_i\subset X$, $K_n[X_i]$ is a copy of $K_{1,t+1}$ centred on a different vertex of the  copy of $K_\tau$ in $K_n[X]$.
Then
$$\val(X,\MS)=|X|-\tau=\tau(t+1-\tau)-\tau=f(\tau)=r(X).$$ 
In addition, since each edge of $X$ is contained in a copy of $K_{1,t}$ in $K_n[X]$, $X$ is a cyclic set in $M$.
We can now apply Lemma \ref{lem:freeseq} to deduce that 
$X\in \bar \MX_0$. Hence (\ref{eq:St}) holds.

%

The proof now breaks into three cases depending on the congruence class of $n$ modulo three.

\paragraph{Case 1: \boldmath $t\equiv 2 \pmod{3}$.} Then $\tau=(2t-1)/3$ and $f(\tau)=\tau^2$.
We first show that
\begin{equation}\label{eq:St-case1}
\mbox{$X$ is independent in $M$ whenever $X\subseteq E$ and $K_n[X]\cong K_{\tau,\tau}$}.
\end{equation}
Suppose for a contradiction that $X$ contains a circuit $C$ of $M$. Then $F=\cl(C)$ is a cyclic flat of $M$ with $r(F)=r(C)<|C|\leq \tau^2=f(\tau)$. Lemma \ref{lem:matroid} now implies that $K_n[F]\cong K_b+\bar K_{n-b}$ for some $1\leq b\leq \tau-1$ and $r(F)=f(b)$. Since $r(C)=r(F)$ this gives $|C|=r(C)+1=f(b)+1$.  The fact that $C\subseteq F$ implies that some set of $b$ vertices in $V(C)$ cover all edges in $C$ and hence that some set of $b$ vertices of $K_{\tau,\tau}$ cover at least $f(b)+1$ edges of $K_{\tau,\tau}$. The fact that any set of $b$ vertices of $K_{\tau,\tau}$ covers at most $\tau b$ edges of $K_{\tau,\tau}$ now gives $\tau b\geq f(b)+1$. This contradicts the fact that 
$\tau b-f(b)-1=b(2\tau+b+1-2t)/2-1<0$ since $1\leq b\leq \tau-1$ and $\tau=(2t-1)/3$. Hence (\ref{eq:St-case1}) holds.

Choose $X_1,X_2\subseteq E$ such that $K_n[X_1]\cong K_\tau+\bar K_\tau\cong K_n[X_2]$ and $K_n[X_1\cup X_2]\cong K_{2\tau}$. Then 
$X_1,X_2\in \bar \X_0$ by (\ref{eq:St}). Since $K_n[X_1\cap X_2]\cong K_{\tau,\tau}$ we can use (\ref{eq:St-case1}) to deduce that $r(X_1\cap X_2)=\tau^2=\rank M=r(X_1)=r(X_2)=r(X_1\cup X_2)$. Since $\bar \MX_0$ is a modular cyclic family, this gives $X_1\cup X_2\in \bar \MX_0$. Lemma \ref{lem:freerank}  now tells us that $r'(X_1\cup X_2)= r(X_1\cup X_2)=\tau^2$. 
On the other hand, the facts that $M'$ is a $K_{1,t}$-matroid and $K_n[X_1\cup X_2]\cong K_{2\tau}$ imply that the closure of $X_1\cup X_2$ in $M'$ is $E$.  This gives $\rank M'=\tau^2=\rank M$ and $M'=M$.

\paragraph{Case 2: \boldmath $t\equiv 0 \pmod{3}$.} Then $\tau=2t/3$ and $f(\tau)=\tau^2-\tau/2$. Let $K_{\tau,\tau}^*$ be the graph obtained from $K_{\tau,\tau}$ by deleting $\tau/2$ disjoint edges.
We first show that
\begin{equation}\label{eq:St-case2}
\mbox{$X$ is independent in $M$ whenever $X\subseteq E$ and $K_n[X]\cong K_{\tau,\tau}^*$}.
\end{equation}
Suppose for a contradiction that $X$ contains a circuit $C$ of $M$. Then $F=\cl(C)$ is a cyclic flat of $M$ with $r(F)=r(C)<|C|\leq \tau^2-\tau/2=f(\tau)$. Lemma \ref{lem:matroid} now implies that $K_n[F]\cong K_b+\bar K_{n-b}$ for some $1\leq b\leq \tau-1$ and $r(F)=f(b)$. Since $r(C)=r(F)$ this gives $|C|=r(C)+1=f(b)+1$.  
Since $C\subseteq F$, some set of $b$ vertices in $V(C)$ cover all edges in $C$ and hence some set of $b$ vertices of $K_{\tau,\tau}^*$ cover at least $f(b)+1$ edges of $K_{\tau,\tau}^*$. 
The fact that any set of $b$ vertices of $K_{\tau,\tau}^*$ covers at most $\tau b$ edges of $K_{\tau,\tau}$ now gives $\tau b\geq f(b)+1$. This contradicts the fact that $\tau b-f(b)-1=b(2\tau+b+1-2t)/2-1<0$ since $1\leq b\leq \tau-1$ and $\tau=2t/3$. Hence (\ref{eq:St-case2}) holds.

Choose $X_1,X_2\subseteq E$ such that $K_n[X_1]\cong K_\tau+\bar K_\tau\cong K_n[X_2]$ and $K_n[X_1\cup X_2]\cong K_{2\tau}$. Then 
$X_1,X_2\in \bar \X_0$ by (\ref{eq:St}). Since $K_n[X_1\cap X_2]\cong K_{\tau,\tau}$ we can use (\ref{eq:St-case2}) to deduce that $r(X_1\cap X_2)=\tau^2-\tau/2=\rank M=r(X_1)=r(X_2)=r(X_1\cup X_2)$. Since $\bar \MX_0$ is a modular cyclic family, this gives $X_1\cup X_2\in \bar \MX_0$. Lemma \ref{lem:freerank}  now tells us that $r'(X_1\cup X_2)= r(X_1\cup X_2)=\tau^2-\tau$. 
On the other hand, the facts that $M'$ is a $K_{1,t}$-matroid and $K_n[X_1\cup X_2]\cong K_{2\tau}$ imply that the closure of $X_1\cup X_2$ in $M'$ is $E$.  This gives $\rank M'=\tau^2-\tau/2=\rank M$ and $M'=M$.

\paragraph{Case 3: \boldmath $t\equiv 1 \pmod{3}$.} Then $\tau=(2t+1)/3$ and $f(\tau)=\tau^2-\tau$. Let $K_{\tau,\tau}^*$ be the graph obtained from $K_{\tau,\tau}$ by deleting a perfect matching from $K_{\tau,\tau}$.  

We first show that
\begin{equation}\label{eq:St-case3}
\mbox{$X$ is independent in  $M$ whenever $X\subseteq E$ and $K_n[X]\cong K_{\tau,\tau}^*$}.
\end{equation}
Suppose for a contradiction that $X$ contains a circuit $C$ of $M$. Then $F=\cl(C)$ is a cyclic flat of $M$ with $r(F)=r(C)<|C|\leq \tau^2-\tau=f(\tau)$. Lemma \ref{lem:matroid} now implies that $K_n[F]\cong K_b+\bar K_{n-b}$ for some $1\leq b\leq \tau-1$ and $r(F)=f(b)$. Since $r(C)=r(F)$ this gives $|C|=r(C)+1=f(b)+1$.  The fact that $C\subseteq F$ implies that some set of $b$ vertices in $V(C)$ cover all edges in $C$ and hence that some set of $b$ vertices of $K_{\tau,\tau}^*$ cover at least $f(b)+1$ edges of $K_{\tau,\tau}^*$. The fact that any set of $b$ vertices of $K_{\tau,\tau}^*$ covers at most $(\tau-1) b$ edges of $K_{\tau,\tau}^*$ now implies that $(\tau-1) b\geq f(b)+1$. This contradicts the fact that
$(\tau-1) b-f(b)-1=b(2\tau+b-1-2t)/2-1<0$ since $1\leq b\leq \tau-1$ and $\tau=(2t+1)/3$. Hence (\ref{eq:St-case3}) holds.

Choose $X_1,X_2\subseteq E$ such that $K_n[X_1]\cong K_\tau+\bar K_\tau\cong K_n[X_2]$ and $K_n[X_1\cup X_2]\cong K_{2\tau}$. Then 
$X_1,X_2\in \bar \X_0$ by (\ref{eq:St}). Since $K_n[X_1\cap X_2]\cong K_{\tau,\tau}$ we can use (\ref{eq:St-case3}) to deduce that $r(X_1\cap X_2)=\tau^2-\tau=\rank M=r(X_1)=r(X_2)=r(X_1\cup X_2)$. Since $\bar \MX_0$ is a modular cyclic family, this gives $X_1\cup X_2\in \bar \MX_0$. Lemma \ref{lem:freerank}  now tells us that $r'(X_1\cup X_2)= r(X_1\cup X_2)=\tau^2-\tau$. 
On the other hand, the facts that $M'$ is a $K_{1,t}$-matroid and $K_n[X_1\cup X_2]\cong K_{2\tau}$ imply that the closure of $X_1\cup X_2$ in $M'$ is $E$.  This gives $\rank M'=\tau^2-\tau=\rank M$ and $M'=M$.
\end{proof}

\medskip

\noindent
{\bf Remark} We can  use the matroid $M_{5,n}$ to show that the converse of Lemma \ref{lem:freeseq} is false. Let 
$\MC_0$ and $\MX_0$ be the families of  non-spanning circuits  and non-spanning cyclic flats of $M_{5,n}$, respectively. Every circuit in $\MC_0$ has rank 4 or 7. Since $M(5,n)$ has only the trivial erection,  $E\in \bar\MX_0$.  We will show that $\val_{\MC_0}(E)>9=r_M(E)$. We can use the proof technique of Lemma \ref{lem:matroid} to show that  there exists a matroid $N$  
on $E(K_n)$ whose cyclic flats are the copies of $K_{1,n-1}$, $K_{2,n-2}$, $K_{3,3}$, $K_{3,n-3}$ and $E$ with ranks  4, 7, 8, 9 and 10, respectively.
Then $N$ and $M_{5,n}$ have the same rank 8 truncation and hence each element of $\MC_0$ is a circuit of $N$. The fact that $\rank N=10$ now implies that $\val_{\MC_0}(E)\geq 10$.

\begin{thebibliography}{99}

\bibitem{BM} {J. E. Bonin  and A. de Mier}, The lattice of cyclic flats of a matroid,
Annals of Combinatorics 12, 2 (July 2008), 155–170.

\bibitem{brakensiek}
Joshua Brakensiek, Manik Dhar, Jiyang Gao, Sivakanth Gopi, Matt Larson,
Rigidity matroids and linear algebraic matroids with applications to matrix completion and tensor codes, arXiv:2405.00778



\bibitem{B} {T.H. Brylawski}, 
 An affine representation for transversal geometries, Studies in Appl. Math.
54 (1975) 143-160.

\bibitem{B86}
T.~Brylawski, Constructions, in Theory of Matroids, N.L.~White, ed., Cambridge Univ.~Press, Cambridge (1986), 127--223.




\bibitem{CJT} K. Clinch, B. Jackson and S. Tanigawa,
Abstract 3-Rigidity and bivariate $C_2^1$-splines I: Whiteley’s Maximality Conjecture, Discrete Analysis 2022:2, 50 pp.

\bibitem{CJT2}
K.~Clinch, B.~Jackson and S.~Tanigawa,
Abstract 3-rigidity and bivariate $C_2^1$-splines II: Combinatorial Characterization, Discrete Analysis,  2022:3, 32 pp.

\bibitem{C} 
H.~Crapo, Erecting geometries, Proc.~2nd Chapel Hill Conf.~on Comb.~Math, 1970, 74--99.



\bibitem{GGJN} G. Grasegger, H. Guler, B. Jackson and  A. Nixon, Flexible circuits, J. Graph Theory, 100 (2022) 315-330.


\bibitem{G91} J. E. Graver, Rigidity matroids, SIAM Journal on Discrete Mathematics, 4, 1991,
355–368.

\bibitem{GSS93} J. E. Graver, B. Servatius, and H. Servatius, Combinatorial rigidity,
Amer. Math. Soc., 1993.

\bibitem{Iz} I. Izmestiev Projective background of the infinitesimal rigidity of frameworks, Geom.
Dedicata, 140 (2009) 183–203.

\bibitem{JT} B. Jackson and S. Tanigawa, Maximal Matroids in Weak Order Posets, J.~Comb.~Theory Ser.~B,  165 (2024), 20-46.

\bibitem{JJT} B. Jackson, T.  Jord{\'a}n, and S. Tanigawa, Combinatorial conditions for the unique completability of low rank matrices, SIAM J. Discrete Math., 28 (2014), 1797-1819.

\bibitem{J} T.~Jord{'a}n. A note on generic rigidity of graphs in higher dimension, Discrete Applied Mathematics, 297 (2021), 97-101.

\bibitem{K} G. Kalai, Hyperconnectivity of graphs. Graphs Combin. 1 (1985) 65–79.

\bibitem{KNN} G. Kalai, E. Nevo, and I. Novik, Bipartite rigidity, Trans. Amer. Math. Soc. 368
(2016) 5515–5545.

 \bibitem{Las} M.~Las Vergnas, On products of matroids, Discrete Math., 36, 49--55, 1981.

\bibitem{L} {L. Lov\'asz}, {Flats in matroids and geometric graphs}, in Cominatorial Surveys (Proc. Sixth
British Combinatorial Conf., Royal Holloway Coll., Egham, ed. P. J. Cameron) Academic Press, 1977, 45-86.

\bibitem{LY}
L.~Lov{\'a}sz and Y.~Yemini , On generic rigidity in the plane, SIAM J.~Algebr.~Discrete Methods, 
3(1), 91--98, 1982.

\bibitem{M} {J. H. Mason}, Glueing matroids together: a stude of Dilworth truncations and matroid analogues of exterior and symmetric powers, in Algebraic Methods in Graph Theory (ed. L. Lovász and V. T. Sós), Coll. Math. Soc. J. Bolyai 25, North-Holland, 1981, 519–561.





\bibitem{N10}
V. H. Nguyen, On abstract rigidity matroids, SIAM Journal on Discrete Mathematics,
24, 2010, 363–369.

\bibitem{O} J.G. Oxley, Matroid theory, 2nd edition, Oxford University Press, 2011.

\bibitem{Servatius}
B.~Servatius, On the two-dimensional generic rigidity matroid and its dual, J.~Comb.~Theory Ser.~B, 53(1), 106--113, 1991.

\bibitem{S} {J. A. Sims}, Some Problems in Matroid Theory. PhD thesis, Linacre
College, Oxford University, 1980.

\bibitem{SC} A. Singer and M. Cucuringu, Uniqueness of low-rank matrix completion by rigidity theory,
SIAM J. Matrix Anal. Appl., 31 (2010) pp. 1621–1641.


\bibitem{W}
W. Whiteley. Some matroids from discrete applied geometry. In Matroid theory (J.E. Bonin,
J.G. Oxley and B. Servatius eds., Seattle, WA, 1995), Contemp. Math., 197, Amer. Math. Soc.,
Providence, RI, 1996: 171–311.
\end{thebibliography}
\end{document}